\documentclass[10pt]{amsart}

\usepackage{times,amsfonts,amsmath,amstext,amsbsy,amssymb,
  amsopn,amsthm,upref,eucal, amscd}
\usepackage[T1]{fontenc}
\usepackage{color}
\usepackage[pdftex]{graphicx}

\newtheorem{theorem}{Theorem}[section]
\newtheorem{lemma}[theorem]{Lemma}
\newtheorem{corollary}[theorem]{Corollary}
\newtheorem{proposition}[theorem]{Proposition}

\numberwithin{equation}{section}

\theoremstyle{definition}
\newtheorem{definition}[theorem]{Definition}

\newtheorem{remark}[theorem]{Remark}

\def\Cset{\mathbb{C}}

 \def\Qset{\mathbb{Q}}
 \def\Rset{\mathbb{R}}
 \def\Zset{\mathbb{Z}}

\def\leq{\leqslant }
\def\geq{\geqslant}
 
\def\A{\mathcal{A}}
\def\Ad{\A_d}

\def\zk{\Zset_{2k}}
\def\zl{\Zset_{\ell}}
\def\bzk{\breve\Zset_{2k}}
\def\bzl{\breve\Zset_{\ell}}
\def\mkl{\mathcal M_{k,\ell}}
\def\hr{H_1(\mkl,\Sigma,\Qset)}
\def\ha{H_1(\mkl,\Qset)}
\def\qe{Q_{even}}
\def\qo{Q_{odd}}
\def\rkl{\mathfrak {R}(2k,\ell)}
\def\rekl{\mathfrak {R}^*(2k,\ell)}
\def\ZP{(\Zset_\varPi)^*}
\def\Zpa{\Zset_{p^A}}
\def\Zpb{\Zset_{p^B}}
\def\Zpc{(\Zset_{p^C})^*}
\def\R{\mathcal{R}}
\def\Rd{\R_d}

\def\D{\mathcal{D}}
\def\Dd{\D_d}
\def\Qa{Q_{\alpha}}
\def\QA{Q'_{\alpha}}
\def\CAD{\Cset^{\Ad}}
\def\To{\longrightarrow}
\begin{document}

\title[Kontsevich--Zorich cocycle over Veech--McMullen family]{The Kontsevich--Zorich cocycle over Veech--McMullen family of symmetric translation surfaces}

\author{Artur Avila, Carlos Matheus and Jean-Christophe Yoccoz}

\date{\today}

\begin{abstract}
We describe the Kontsevich--Zorich cocycle over an affine invariant orbifold coming from a (cyclic) covering construction inspired by works of Veech and McMullen. In particular, using the terminology in a recent paper of Filip, we show that all cases of Kontsevich--Zorich monodromies of $SU(p,q)$ type are realized by appropriate covering constructions.

\end{abstract}
\maketitle
%\pageheight{8.5truein}
%\pagewidth{6.5truein}

\tableofcontents

\section{The Veech--McMullen family of symmetric translation surfaces}\label{secdef}
%%%%%%%%%%%%%%%%%%%%%%%%%%%%%%%%%%%%%%%%%%%%%%%%%%%%%%%%%%%%%%%%%%%%%%%%%%%%%%%%

\subsection{Definition and Notations}\label{ssdefnot}

Let $k$ be a positive integer, and let $\ell$ be an integer at least equal to $3$. We denote by $R$ the rotation of $\Cset$ centered at $0$ of angle $2\pi/ \ell$, by $S$ the symmetry $z \mapsto -\bar z$ w.r.t. the imaginary axis. 

\par

 We write $\Zset_m$ for the standard cyclic group with $m$ elements and $\breve \Zset_m$ for the $\Zset_m$-homogeneous
space of pairs of consecutive elements of $\Zset_m$. 

\par

 Let $Q_{even} \subset \Cset$ be the closed regular polygon whose vertices are the roots of unity of order $\ell$. Let $Q_{odd} := S(Q_{even}) =-Q_{even}$ . 

\par

Consider $k$ copies  of $Q_{even}$, indexed by the even elements of $\Zset_{2k}$, and $k$ copies  of $Q_{odd}$,  indexed by the odd elements of $\Zset_{2k}$. 

\par

The vertices of $\qe$ and $\qo$ are indexed by

$$ A(even,j) = R^j(1), \quad A(odd,j) = -A(even,j),\quad \forall j \in \zl.$$

\smallskip

For  $j' = (j,j+1) \in \bzl$, $\epsilon \in \{even,odd\}$, we denote by $M(\epsilon,j')$ the midpoint between $A(\epsilon,j)$ and $A(\epsilon,j+1)$, by $\aleph^+(\epsilon,j)$ the oriented segment from $A(\epsilon,j)$ to $M(\epsilon,j')$, and by $\aleph^-(\epsilon,j+1)$ the oriented segment from $A(\epsilon,j+1)$ to $M(\epsilon,j')$.

\par

For $j \in \zl$ and $i \in \zk$, with parity $\epsilon$, we denote by $A(i,j),\, M(i,(j,j+1)), \, \aleph^\pm (i,j)$ the copies in  $Q_i$ of $A(\epsilon,j),\, M(\epsilon,(j,j+1), \aleph^\pm(\epsilon,j)$.

\par

\begin{definition} 
The  translation surface $\mathcal M_{k,\ell}$ is obtained from the disjoint union of the $Q_i$, $i\in \zk$ by identifying through the appropriate translation, for each $i\in \zk, j \in \zl$, the segment $\aleph^+(i,j)$ with the segment $\aleph^-(i+1,j+1)$. We denote by $\aleph ((i,i+1),(j,j+1))$ the image of these segments in $\mkl$.
\end{definition}

%%%%%%%%%%%%%%%%%%%%%%%%%%%%%%%%%%%%%%%%%%%%%%%%%%%%%%%%%%%%%%%%%%%%%%%%%%%%%%%%

\subsection{Basic properties and symmetries of the translation surface $\mkl$}\label{ssbasicprop}

We start by computing the ramification at the singular set, associated to the $M(i,(j,j+1))$ and $A(i,j)$.

\par
For each $j' \in \bzl$, the points $M(i,j'), i \in \zk$ are identified into a single point $M(j')$ on $\mkl$ where the total angle is $2\pi k$. 
\par
On the other hand, when rotating counterclockwise around $A(i,j)$, a sector of angle 
$\frac {\pi(\ell -2)}{\ell}$ in $Q_i$ is followed  by a sector of the same angle at $A(i-1,j-1)$ in $Q_{i-1}$. The points of $\mkl$ corresponding to the $A(i,j)$ are therefore naturally indexed by the orbits of the transformation $(i,j) \rightarrow (i-1,j-1)$ on $\zk \times \zl$. We denote  by $A(\varDelta)$ the point of $\mkl$ associated to an orbit $\varDelta$. The number of such orbits is the greatest common divisor $\varpi$ of $2k$ and $\ell$. The total angle at such a point $A(\varDelta)$ is 
$\frac {2\pi k(\ell -2)}\varpi$.

\smallskip

\par
We denote by $\Sigma$ the set of marked points $M(j'), A(\varDelta)$ of $\mkl$, of cardinality $\ell +\varpi$.

\smallskip

The genus of $\mkl$ is thus given by
$$g= \ell k +1-k- \frac 12(\ell +\varpi).$$
Observe that $\ell$ and $\varpi$ have the same parity.

\medskip

\begin{remark}
The translation surface $\mathcal M_{1,\ell}$ has been first studied by Veech \cite{Ve89}. He shows that it is a Veech surface and that the image of the Veech group in $PSL(2,\Rset)$ is the lattice generated by $R$ and the parabolic element 
$$ \left ( \begin{array}{cc} 1&0 \\ 2 \cot \frac \pi \ell & 1 \end{array} \right ) .$$
It follows easily that the the subset $\Sigma_M \subset \mathcal M_{1,\ell}$ consisting of the $\ell$ points $M_{j'}, \, j' \in \bzl$ is invariant under the group of affine homeomorphisms of $\mathcal M_{1,\ell}$. The same is true of the the subset $\Sigma_A$ consisting of the one (if $\ell$ is odd) or two (if $\ell$ is even) points $A(\varDelta)$.
\end{remark}

\begin{remark}
By the previous remark, the image in $\mathcal M_{1,\ell}$ of the ramification set of the natural projection $\mkl \to \mathcal M_{1,\ell}$ is invariant under the group of affine homeomorphisms of $\mathcal M_{1,\ell}$. It follows from a result of Gutkin--Judge \cite{GJ} that $\mkl$ is a Veech surface.
\end{remark}

\subsection{Covers of hyperelliptic components of strata} Algebraically, the translation surface $\mathcal{M}_{1,\ell}$ corresponds to the Riemann surface $y^2=(x-x_1)\dots(x-x_{\ell})$ together with the holomorphic one-form $c dx/y$ for \emph{appropriate} choices of $\ell$ distinct points $x_1,\dots, x_{\ell}\in\mathbb{C}$ and a constant $c\in\mathbb{C}^*$: see Veech \cite{Ve89}. More generally, $\mkl$ is the covering of $\mathcal{M}_{1,\ell}$ given by the Riemann surface $y^{2k} = (x-x_1)\dots(x-x_{\ell})$ and the holomorphic one-form $c dx/y^k$: see McMullen \cite{McM}. For this reason, we define the Veech--McMullen family $\mathcal{F}_{k,\ell}$ of translation surfaces the Riemman surfaces $y^{2k}=(x-x_1)\dots(x-x_{\ell})$ equipped with $c dx/y^k$ for \emph{arbitrary} choices of $\ell$ distinct points $x_1,\dots, x_{\ell}\in\mathbb{C}$ and constants $c\in\mathbb{C}^*$. 

The Veech--McMullen family $\mathcal{F}_{1,\ell}$ for $\ell$ odd, resp. $\ell$ even, is the hyperelliptic component of the stratum $\mathcal{H}(\underbrace{(\ell-2-\varpi)/\varpi, \dots, (\ell-2-\varpi)/\varpi}_{\varpi})$ of translation surfaces with $\varpi$ conical singularities with total angle $\frac {2\pi k(\ell -2)}\varpi$: see \cite{KZ03}. 

In general, the Veech--McMullen family $\mathcal{F}_{k,\ell}$ is an affine suborbifold of the stratum $\mathcal{H}(\underbrace{k-1,\dots, k-1}_{\ell}, \underbrace{k(\ell-2-\varpi)/\varpi, \dots, k(\ell-2-\varpi)/\varpi}_{\varpi})$ given by a covering construction. Therefore, the Kontsevich--Zorich cocycle over the Teichm\"uller geodesic flow on $\mathcal{F}_{k,\ell}$ is coded by the hyperelliptic Rauzy diagrams with arrows decorated by certain matrices (describing actions on the homology of canonical translation surfaces in $\mathcal{F}_{k,l}$). 

In particular, following the discussion in our previous paper \cite{AMY-hyp}, one can associate a \emph{Rauzy--Veech group} $RV(k,\ell)$ to $\mathcal{F}_{k,\ell}$. By definition, the matrices in $RV(k,\ell)$ preserve the natural symplectic intersection form on the absolute homology of the translation surfaces in $\mathcal{F}_{k,\ell}$. 

\begin{remark}\label{r.AMY-hyp} In this notation, the main result from our previous paper \cite{AMY-hyp} asserts that $RV(1,k)$ is naturally isomorphic to an explicit finite-index subgroup of the integral symplectic group $Sp(2g,\mathbb{Z})$ (where $g$ is the genus of $\mkl$).  
\end{remark}

\subsection{Statement of the main result} In this paper, we study the structure of the Rauzy--Veech groups of $\mathcal{F}_{k,\ell}$.

\begin{theorem}\label{t.A} The real Hodge bundle over $\mathcal{F}_{k,\ell}$ decomposes into a direct sum $H_1\oplus\dots\oplus H_k$ of flat subbundles $H_r$ associated to the eigenspaces of the generator of the deck group of $\mathcal{M}_{k,\ell}\to\mathcal{M}_{1,\ell}$ (cf. \eqref{e.Hr} and \eqref{e.Hk} below). The Rauzy--Veech group of $\mathcal{F}_{k,\ell}$ respects this decomposition. Moreover, if one denotes by $RV(k,\ell)|_{H_r}$ the group associated to the restrictions to $H_r$ of the matrices in $RV(k,\ell)$, then: 
\begin{itemize}
\item[(a)] $RV(k,\ell)|_{H_k}$ is naturally isomorphic to $RV(1,\ell)$; thus, $RV(k,\ell)|_{H_k}$ is isomorphic to a finite-index subgroup of $Sp(2g,\mathbb{Z})$;  
\item[(b)] for each $0<r<k$, the symplectic intersection form on $H_r$ induces a Hermitian form $Q_{r/2k}$ of signature $(\lceil \ell (r/2k) -1 \rceil, \lceil \ell(1-(r/2k)) -1 \rceil)$ which is preserved by $RV(k,\ell)|_{H_r}$; furthermore, 
\begin{itemize}
\item if $\ell (r/2k)<1$ and $r/2k \neq 1/6, 1/4$, then $RV(k,\ell)|_{H_r}\cap SU(Q_{r/2k})$ is dense in $SU(Q_{r/2k})$ (for the usual topology); 
\item if $\ell (r/2k)\notin\mathbb{Z}$ and $r/2k\neq 1/6, 1/4, 1/3$, then $RV(k,\ell)|_{H_r}\cap SU(Q_{r/2k})$ is Zariski dense in $SU(Q_{r/2k})$. 
\end{itemize}
\end{itemize} 
\end{theorem}

\begin{remark} Actually, our discussion in Section \ref{s.matrices} provides a precise version of Theorem \ref{t.A}: in particular, we compute the Zariski closure of $RV(\ell, k)|_{H_r}\cap SU(Q_{r/2k})$ in the exceptional cases $r/2k = 1/6, 1/4, 1/3$. However, we have not included all possibilities in  Theorem \ref{t.A} in order to get a ``cleaner'' statement. 
\end{remark}

This result provides explicit examples showing that all cases of $SU(p,q)$ Kontsevich--Zorich monodromies discussed in Filip's paper \cite{Fi} actually occur. 

A direct consequence of Theorem \ref{t.A} and the simplicity criterion of Avila--Viana \cite{AV} (as stated in Subsection 2.5 of \cite{MMY}) is: 

\begin{corollary}\label{c.A} The Lyapunov exponents of the restriction of the Kontsevich--Zorich cocycle over $\mathcal{F}_{k,\ell}$ to $H_r$ are ``simple'' in the sense that:
\begin{itemize}
\item they have multiplicity one when $r=k$;
\item for $0<r<k$, $\ell(r/2k)\notin\mathbb{Z}$ and $r/2k\neq 1/6, 1/4, 1/3$, the multiplicity of all non-zero Lyapunov exponents is one and there are exactly $\lceil \ell(1-(r/2k)) -1 \rceil-\lceil \ell (r/2k) -1 \rceil$ vanishing Lyapunov exponents. 
\end{itemize}
\end{corollary}

\begin{proof} As it is explained in \cite[\S 7.3]{AV}, the Kontsevich--Zorich cocycle over $\mathcal{F}_{k,\ell}$ is coded by a locally constant cocycle whose supporting monoid $\mathcal{B}$ contains all matrices of the form $B_{\gamma}BB_{\gamma}$, where $B_{\gamma}$ is a fixed Kontsevich--Zorich matrix and $B$ are the Kontsevich--Zorich matrices associated to all oriented loops based at a fixed vertex of the Rauzy diagram. Since the group generated by all such matrices $B$ is $RV(k,\ell)$, it follows from Theorem \ref{t.A} that the restriction of the supporting monoid $\mathcal{B}$ to each $H_r$ is pinching and twisting in the sense of \cite[\S 2.5]{MMY}. The desired result now follows from \cite[Theorem 2.17]{MMY}.
\end{proof}

\subsection{Organization of the paper} In Section \ref{s.homology}, we study a decomposition $H_1\oplus\dots\oplus H_k$ of the first absolute homology group of $\mkl$. In Section \ref{s.RV}, we describe the restrictions $RV(k,\ell)|_{H_r}$ of the Rauzy--Veech group $R(k,\ell)$ to the summands of the decomposition $H_1(\mkl, \mathbb{R}) = H_1\oplus\dots\oplus H_k$ in terms of complex matrices on the vector space $\mathbb{C}^{\ell}$ equipped with adequate hermitian forms. In Section \ref{s.RV}, we reduce the proof of Theorem \ref{t.A} to the investigation of certain groups of complex matrices. Finally, we analyse in Section \ref{s.matrices} the relevant groups of matrices in order to establish Theorem \ref{t.A}. 

\begin{remark} We hope that the arguments in this paper might be useful to study the question of non-continuity of the central Oseledets subspaces of the Kontsevich--Zorich cocycle. 
\end{remark}

\section{Homology groups}\label{s.homology}

\subsection{Subgroups of affine diffeomorphisms} 
 
We now define a finite subgroup $G$ of order $4k\ell$ of the group of affine diffeomorphisms of $\mkl$. 
  \par
The group $\zk$ acts by direct affine diffeomorphisms on $\mkl$: the element $r \in \zk$
 sends $Q_i$ onto $Q_{i+r}$ with derivative ${\rm id}$ if $r$ is even, ${\rm -id}$ if $r$
  is odd. It sends $A(i,j )$ to $A(i+r,j )$, $\aleph ((i,i+1),(j,j+1))$ to 
  $\aleph ((i+r,i+r+1),(j,j+1))$ and fixes each $M(j')$,
 
\par
\smallskip
 
The group $\zl$ acts on $\mkl$ by direct affine diffeomorphisms, the derivative of the action of $s\in \zl$ is the rotation $R^s$. This action preserves each $Q_i$, sending $A(i,j)$ to $A(i,j+s)$, $M((j,j+1))$ to $M((j+s,j+s+1))$, $\aleph ((i,i+1),(j,j+1))$ to $\aleph ((i,i+1),(j+s,j+s+1))$.
\par
\smallskip

This two actions commute  and they combine to define an action of the product group $\zk \times \zl$ on $\mkl$.

\par
\smallskip

We now define an affine involution $\sigma$ of $\mkl$ whose derivative is $S$. For $i \in \zk$, $\sigma$ sends $Q_i$ onto $Q_{1-i}$, $M((j,j+1))$ to $M((-j-1,-j))$, $A(i,j)$ to $A(1-i,-j)$, \linebreak $\aleph ((i,i+1),(j,j+1))$ to $\aleph ((-i,-i+1),(-j-1,-j))$.

\par
\smallskip

The involution $\sigma$ conjugates the  action of every element $(r,s) \in \zk \times \zl$ to the action of $(-r,-s)$. Combining the action of $\sigma$ and the action of $ \zk \times \zl$ defines  an action of a group $G$ of order $4k\ell$. It is the usual dihedral group of order $4k\ell$ when $\varpi=1$.
The group $G$ is the group of permutations of $\bzk \times \zl$ of the form $(i',j) \rightarrow (r,s)\pm (i',j)$, for $r\in \zk$, $s \in \zl$.

%%%%%%%%%%%%%%%%%%%%%%%%%%%%%%%%%%%%%%%%
%%%%%%%%%%%%%%%%%%%%%%%%%%%%%%%%%%%%%%%%

\subsection{Conjugacy classes in $G$}\label{ssconjclas}

 Let $1_{2k}$, $1_{\ell}$ be the standard generators of the corresponding cyclic groups.

 \begin{itemize}
\item Assume first that $\ell$ is odd. There are $k\ell +1$ conjugacy classes of $G$ contained in $\zk \times \zl$,
more precisely $2$ elements of order $2$ and $k\ell -1$ pairs of  distinct elements inverse to each other. The other $2$ conjugacy classes are those of $\sigma$ and $\sigma 1_{2k}$ and have size $k \ell$.
\item Assume now  that $\ell$ is even. There are $k\ell +2$ conjugacy classes of $G$ contained in $\zk \times \zl$,
more precisely $4$ elements of order $2$ and $k\ell -2$ pairs of  distinct elements inverse to each other. The other $4$ conjugacy classes are those of $\sigma$, $\sigma 1_{2k}$, $\sigma 1_{\ell}$, $\sigma 1_{2k} 1_{\ell}$ and have size $ \frac 12 k \ell$.
\end{itemize}
%%%%%%%%%%%%%%%%%%%%%%%%%%%%%%%%%%%%%%%%%%%%%%%%%%%%%%%%%%%%%%%%%%%%%%%%%%%%%%%%

\subsection{Irreducible representations of $G$ over $\Rset$ or $\Cset$}\label{ssirrC}

The irreducible representations of $G$ over $\Cset$ are all defined over $\Rset$ and have dimension $1$ or $2$.

\par
\smallskip

When $\ell$ is odd, there are $4$ nonequivalent $1$-dimensional representations of $G$, sending $1_{\ell}$ to $1$ and $1_{2k}$ and $\sigma$ to $\pm 1$ (independently).

\par
\smallskip

When $\ell$ is even, there are $8$ nonequivalent $1$-dimensional representations of $G$, sending $1_{\ell}$ , $1_{2k}$ and $\sigma$ to $\pm 1$ (independently).

\par
\smallskip

To parametrize the $2$-dimensional representations, we define $\mathfrak R(2k,\ell)$ to be the quotient of $\zk \times \zl$ by $\{ \pm 1 \}$ and by $\rekl$ the subset of $\rkl$ associated to elements $(r,s)  \in \zk \times \zl$ of order $>2$ . 
The cardinality of $\mathfrak R(2k,\ell)$ (resp. $\rekl$) is equal to $k\ell +1$ (resp. $k\ell-1$) if $\ell$ is odd, to $k\ell +2$ (resp. $k\ell-2$) is $\ell$ is even.

\par
\smallskip

For $(r,s) \in \zk \times \zl$, the representation $\pi_{r,s}$ sends $\sigma$ to $S$, $1_{2k}$  to a rotation of angle $\pi \frac rk$ and  $1_{\ell}$ to a rotation of angle $2\pi \frac s{\ell}$. 
\par
\smallskip

The character $\chi_{r,s}$ of $\pi_{r,s}$ vanishes on $G - (\zk \times \zl)$ and its values on $\zk \times \zl$ are given by
$$ \chi_{r,s}(i,j)= 2 \cos 2\pi ( \frac {ri}{2k} + \frac {sj}{\ell}).$$

\par 
\smallskip

When $(r,s)$ has order $>2$, the representation $\pi_{r,s}$ is irreductible  and is conjugated by $S$ to $\pi_{-r,-s}$. 

\par
\smallskip

When $(r,s)$ has order $2$, the representation $\pi_{r,s}$ split into two representations $\pi_{r,s}^+$ (with $\chi_{r,s}^+(\sigma) =1$) and $\pi_{r,s}^-$ (with $\chi_{r,s}^-(\sigma) =-1$) of dimension $1$. The isomorphism classes of irreducible  representations are thus parametrized by $\rekl$. 

%%%%%%%%%%%%%%%%%%%%%%%%%%%%%%%%%%%%%%%%%%%%%%%%%%%%%%%%%%%%%%%%%%%%%%%%%%%%%%%%

\subsection{Irreducible representations of $G$ over $\Qset$}\label{ssirrQ}

Let $\varPi:= 2k\ell/\varpi$ be the least common multiple of $2k$ and $\ell$, which is 
also the least common multiple of the elements of $G$. By a general result of Brauer (see \cite[Theorem 24, Section 12.3]{Serre}), which is easily checked  in the case of $G$, all irreducible representations of $G$ over $\Cset$ are actually defined over the cyclotomic field $\Qset(\varPi)$. To see how to group together the $\pi_{r,s}$ in order to obtain the irreducible representations of $G$ over $\Qset$, consider the action of the Galois group $(\Zset_\varPi)^*$ of $\Qset(\varPi)$ over $\Qset$ on $\rkl$ defined by 
$$ t.(r,s) = (tr,ts).$$
The group $(\Zset_\varPi)^*$ also acts on $G$ through $t.g:= g^t$. Observe that we have, for all $(r,s) \in \rkl$, $g \in G$, $t \in (\Zset_\varPi)^*$
$$ t. \chi_{r,s}(g) = \chi_{r,s} (t.g) =  \chi_{t.(r,s)} (g).$$

\par
\smallskip

The $1$-dimensional representations of $G$ are all defined over $\Qset$. Similarly, each of the points in $\rkl \setminus \rekl$ is fixed by $\ZP$.

\par 
\smallskip

By \cite[Theorem 29, section 13.1]{Serre}, we conclude that the irreducible representations of $G$ over $\Qset$ of dimension $>1$ are parametrized by the orbits of the action of $(\Zset_\varPi)^*$ on $\rekl$: for an orbit $\mathcal O$, the associated representation is 
$$ \pi_{\mathcal O} := \oplus_{(r,s) \in \mathcal O} \;\pi_{r,s}.$$

\begin{remark}
To compute the orbits of the action of $(\Zset_\varPi)^*$ on $\rkl$, it is sufficient to consider the action over $\zk \times \zl$, as $-1 \in \ZP$. Then one uses the Chinese remainder theorem to split $\varPi$ into prime powers. Write $\varPi = \prod_p  p^{C_p}$, $2k = \prod_p  p^{A_p}$, $\ell = \prod_p p^{B_p}$, with $C_p = \max (A_p,B_p)$ for each prime $p$. The action of $\ZP$ on $\zk \times \zl$ is the product 
of the actions of $\Zpc$ over $\Zpa \times \Zpb$ (with $A=A_p, B=B_p, C=C_p$).
Let $ (r,s) \in \Zpa \times \Zpb$. Let $p^a$ (resp. $p^b$) be the order of $r$ in $\Zpa$
(resp. of $s$ in $\Zpb$), with $0 \leq a \leq A$, $0 \leq b \leq B$. For any $(r',s')$ in the orbit of $(r,s)$, the order of $r'$ (resp. $s'$) is also $p^a$ (resp. $p^b$). When
$\min (a,b) =0$, there is exactly  one orbit with these orders. When $ \min (a,b) >0$, the stabilizer of $(r,s)$ as above has cardinality equal to $p^{C - \max (a,b)}$, hence the corresponding orbit has cardinality equal to $p^{\max (a,b) -1} (p-1)$. As the number of pairs $(r,s)$ with orders $(p^a, p^b)$ is equal to $p^{a+b-2} (p-1)^2$, the number of orbits associated to these values of $a,b$ is equal to $p^{\min (a,b)-1} (p-1)$.
\end{remark}

%%%%%%%%%%%%%%%%%%%%%%%%%%%%%%%%%%%%%%%%%%%%%%%%%%%%%%%%%%%%%%%%%%%%%%%%%%%%%%%%

\subsection {Decomposition of the first relative homology group}\label{ssrelhom}

The classes of the oriented segments $\aleph (i',j')$ ( $i' \in \bzk$, $j' \in \bzl$) obviously span the first relative homology group $\hr$.
\par
\smallskip

Going around the boundary of $Q_i$ gives the relation 
$$ \sum_{\bzl} \aleph ((i-1,i),j') =  \sum_{\bzl} \aleph ((i+1,i)),j'), \quad \forall i \in \zk,$$
providing $2k-1$ independent relations between the $\aleph(i',j')$. 
As
$$2k\ell - (2k-1) = 2g + (\# \Sigma -1) = \dim \hr ,$$
these are the only relations between the $\aleph(i',j')$.

\par
\smallskip

Denote by $(e_i)$, $(E_{i',j'})$ the canonical bases of $\Qset ^{\zk}$, $\Qset ^{\bzk \times \bzl}$ respectively. We equip $\Qset$, $\Qset ^{\zk}$, $\Qset ^{\bzk \times \bzl}$
with  structures of $G$-module by defining

$$ 1_k . x = 1_\ell .x =x, \quad \sigma.x = -x,\quad \forall x \in \Qset, $$

$$ 1_k .e_i = e_{i+1} , \quad 1_{\ell}.e_i = e_i, \quad \sigma.e_i = -e_{1-i}, \quad \forall i \in \zk,$$

$$ 1_k . E_{i',j'} = E_{1+i',j'}, \; 1_{\ell}.E_{i',j'} = E_{i',1+j'}, \; \sigma.E_{i',j'} = E_{1-i',-j'}, \; \forall (i',j') \in \bzk \times \bzl.$$

We have then an exact sequence of $G$-modules.
\begin{equation}\label{e.exact-sequence}
0 \To \Qset  \To \Qset ^{\zk} \To \Qset ^{\bzk \times \bzl}  \To \hr \To  0.
\end{equation}

The maps in this exact sequence are as follows.
 The map from $\Qset$ to  $\Qset ^{\zk}$ sends $1$ to $\sum_{i} e_i$. 
The map from $\Qset ^{\zk}$ to $\Qset ^{\bzk \times \bzl}$ sends $e_i$ to $\sum_{j'} (E_{\{i,i+1\},j'} -E_{\{i-1,i\},j'})$. The map from 
$ \Qset ^{\bzk \times \bzl}$ to $\hr$ sends $E_{i',j'}$ to the class of $\aleph(i',j')$.
\medskip

\par
From \eqref{e.exact-sequence}, we deduce the character $\chi_{rel} $ of the $G$-module $\hr$. Firstly, for $(i,j) \in \zk \times \zl$, we have 
$$ \chi_{rel} (i,j) = 1 - 2k \delta_{0i} + 2k\ell \delta_{0i} \delta_{0j}.$$

\begin{itemize}
\item
When $\ell$ is odd,  $\chi_{rel}$ is equal to $1$ everywhere on  $G- (\zk \times \zl)$. One obtains

$$ \chi_{rel} = 1 + \sum_{(r,s) \in \rekl,\, s\ne 0} \chi_{r,s} .$$

\item
When $\ell$ is even, the values of $\chi_{rel}$ on the conjugacy classes of $\sigma$, $\sigma 1_{2k}$, 
$\sigma 1_{\ell}$, $\sigma 1_{2k} 1_{\ell}$ are respectively $-1$ ,$+1 $ ,$+3 $ ,$+1 $. One obtains
$$ \chi_{rel} = 1 + \chi_+ + \chi_- + \sum_{(r,s) \in \rekl,\, s\ne 0} \chi_{r,s},$$
where $\chi_+$ (resp. $ \chi_-$) is the  $1$-dimensional character with value $-1$ at $\sigma$ and $1_\ell$ and value $+1$ (resp. $-1$)  at $1_{2k}$.
\end{itemize}

%%%%%%%%%%%%%%%%%%%%%%%%%%%%%%%%%%%%%%%%%%%%%%%%%%%%%%%%%%%%%%%%%%%%%%%%%%%%%%%%

\subsection {Decomposition of the first absolute homology group}

The exact sequence of $G$-modules for relative homology reads
$$0 \rightarrow \ha\rightarrow \hr \rightarrow H_0(\Sigma, \Qset) \rightarrow \Qset \rightarrow 0.$$

This gives 
$$\chi_{ab} = \chi_{rel} - \chi_{\Sigma} +1,$$
where $\chi_{ab}$, $\chi_{\Sigma}$ are the characters of $\ha$, $H_0(\Sigma, \Qset)$ respectively. 
\par
\smallskip

Each of the subsets $\Sigma_M := \{M_{j'}, \; j' \in \bzl \}$ and $\Sigma_A := \{A(\varDelta) , \; \varDelta \in \Zset_\varpi \}$ is invariant under $G$, hence the character $\chi_{\Sigma}$ splits as $\chi_{\Sigma_M} + \chi_{\Sigma_A}$.

\medskip

 The character $\chi_{\Sigma_M}$ satisfies $\chi_{\Sigma_M} (i,j) = \ell \delta_{0j}$.
 
 \par
 \smallskip
 
When $\ell$ is odd, it is equal to  $1$ on all of $G- (\zk \times \zl)$. This gives in this case
$$ \chi_{\Sigma_M} = 1 + \sum_{(0,s) \in \rekl} \chi_{0,s}.$$

\par 
\smallskip

When $\ell$ is even, $\chi_{\Sigma_M}$ takes  the value $0$ on the conjugacy classes of $\sigma$ and $\sigma 1_{2k}$, and the value $2$ 
on the conjugacy classes of $\sigma 1_{\ell}$ and $\sigma 1_{2k} 1_{\ell}$. This gives in this case
$$ \chi_{\Sigma_M} = 1 + \chi_+ + \sum_{(0,s) \in \rekl} \chi_{0,s}.$$
\par
\bigskip

The character $\chi_{\Sigma_A} $ satisfies  $\chi_{\Sigma_A} (i,j)= \varpi$  if $i,j$ are congruent modulo $\varpi$, to $\chi_{\Sigma_A} (i,j)=0$ otherwise. 

\par
\smallskip

When $\ell$ is odd, it is equal to $1$ on all of $G- (\zk \times \zl)$. This gives

$$\chi_{\Sigma_A} = 1 + \sum_{(r,s) \in \rekl, \frac r{2k} + \frac s{\ell} \in \Zset} \chi_{r,s}.$$

\par
\smallskip 

When $\ell$ is even, $\chi_{\sigma_A}$ takes the  value $0$ on the conjugacy classes of $\sigma$ and $\sigma 1_{2k} 1_{\ell}$, and the value $2$ 
on the conjugacy classes of $\sigma 1_{\ell}$ and $\sigma 1_{2k} $. This gives 

$$\chi_{\Sigma_A} = 1 + \chi_- + \sum_{(r,s) \in \rekl, \frac r{2k} + \frac s{\ell} \in \Zset} \chi_{r,s}.$$

In the formula $\chi_{ab} = \chi_{rel} -  \chi_{\Sigma_M}- \chi_{\Sigma_A} +1$ we have now computed all the terms in the right-hand side. We obtain

$$\chi_{ab}= \sum_{(r,s) \in \rkl, r \ne 0, s\ne 0,\frac r{2k} + \frac s{\ell} \notin \Zset } \chi_{r,s}.$$

\begin{remark}
With the conditions $r \ne 0, s\ne 0,\frac r{2k} + \frac s{\ell} \notin \Zset $, there is no difference between $\rkl$ and $\rekl$.

\end{remark}

From this character formula, we get
\begin{eqnarray*}
 H_1(\mkl,\Rset) &=& \oplus _{(r,s) \in \rkl, r \ne 0, s\ne 0,\frac r{2k} + \frac s{\ell} \notin \Zset }\; \pi_{r,s}\\
         &=& \oplus_{0<r \leq k} H_r,
 \end{eqnarray*}
with 
\begin{equation}\label{e.Hr}  H_r:= \oplus_{0<s<\ell,\frac r{2k} + \frac s{\ell} \ne 1 }\; \pi_{r,s}
\end{equation}
for $0 <r <k$ and
\begin{equation}\label{e.Hk}H_k:=  \oplus_{0<s<\ell/2} \; \pi_{k,s}.
\end{equation}
 
 \medskip
 
 \begin{proposition}
 The subrepresentation $H_r$ is defined over $\Qset$ if and only if $r/2k$ is equal to $1/6,\,1/4, \, 1/3$ or $1/2$.
 \end{proposition}
 
 \begin{proof}
  Indeed, for $0 <r \leq k$, let $\mathfrak R_r$ be the image in $\rekl$ of the set $\{ (r,s) \vert \, 0<s<\ell,\frac r{2k} + \frac s{\ell} \ne 1 \}$.  By subsection \ref{ssirrQ}, the subrepresentation $H_r$ is defined over $\Qset$ iff $\mathfrak R_r$ is invariant under the action of $(\Zset_\varPi)^*$. As the subsets $ \zk \times \{0\}$ and $\{(r,s) \vert \frac r{2k} + \frac s{\ell} \in \Zset \}$ of $\zk \times \zl$ are invariant under the action of $(\Zset_\varPi)^*$, $\mathfrak R_r$ is invariant under the action of $(\Zset_\varPi)^*$ iff the subset $\{ \pm r \}$ of $\zk$ is invariant under the action of $(\Zset_\varPi)^*$, which has the same orbits in $\zk$ that the action of $(\Zset_{2k})^*$. But the only integers $n >1$ such that 
  $(\Zset_{n})^* = \{ \pm 1 \}$ are $2,3,4,6$. This proves the proposition.
 \end{proof}
 
\begin{remark}\label{r.Hk} The subspace $H_k$ is identified with the $H_1(\mathcal{M}_{1,\ell}, \mathbb{R})$ under the natural projection $\mkl\to\mathcal{M}_{1,\ell}$. 

In general, the subspaces $H_r$ are given by the eigenspaces of the generator of the group of deck transformations of $\mkl\to\mathcal{M}_{1,\ell}$. In particular, the decompostion $ H_1(\mkl,\Rset) = \oplus_{0<r \leq k} H_r$ can be extended to all translation surfaces of the Veech--McMullen family $\mathcal{F}_{k,\ell}$. 
\end{remark}
 
%%%%%%%%%%%%%%%%%%%%%%%%%%%%%%%%%%%%%%%%%%%%%%%%%%%%%%%%%%%%%%%%%%%%%%%%%%%%%%%%
 
 \subsection{Relation with the symplectic intersection form}

Let $\chi_{\sigma}$ be the $1$-dimensional character on $G$ with kernel $\zk \times \zl$. The intersection form $\omega$ on $\ha$
satisfies, for any $g \in G$, $v,w \in \ha$
$$\omega(g.v,g.w)= \chi_{\sigma}(g) \omega(v,w).$$

\medskip

\begin{proposition}
The $2$-dimensional summands in the decomposition of the $G$-module $\ha$ in the last subsection are mutually $\omega$-orthogonal.
\end{proposition}
\begin{proof}
Let $E,F$ be two such distinct summands.The intersection form defines an homomorphism $u$ from $E$ to the dual $F^*$ of $F$. We have to prove that $u=0$. If $v \in E$ belongs to the kernel of $u$, then the same is true for $g.v$, for any $g \in G$. As $E$ is irreducible, this proves that $u$ is either invertible or equal to $0$. If $u$ is invertible,
the formula above shows that it is an isomorphism of $G$-modules from $E$ to the tensor product of the contragredient representation of $F$ by $\chi_{\sigma}$. But this tensor product is isomorphic as $G$-module to $F$ itself. As $E,F$ are distinct, we conclude that $u$ cannot be invertible, hence $u=0$.
\end{proof}

For $0<r \leq k$, let $H_r$ be the sum of the summands in the decomposition of $\ha$ with characters $\chi_{r,s}$, $s$ varying according to the prescription above. Assume now that $0<r<k$. We equip $H_r$ with the complex structure\footnote{It comes from the fact that $H_r$ is  associated to eigenspaces of the eigenvalues $\exp(\pm\frac{2\pi i r}{2k})$ for the action on complex cohomology of the generator $1_{2k}$ of the deck group of $\mathcal{M}_{k,\ell}\to\mathcal{M}_{1,\ell}$ (cf. Remark \ref{r.Hk}).} such that
$$ 1_{2k} .v = \exp(i\pi\frac rk) v.$$ 
The formula
$$\langle v,v \rangle:= (\sin(\pi \frac rk))^{-1} \omega(1_{2k} .v,v)$$
defines an hermitian\footnote{In the literature, this hermitian form is called \emph{Hodge form}.} form on $H_r$ whose imaginary part is $\omega$. 

The signature of this hermitian form is calculated in Subsection \ref{ss.signature}. 

%%%%%%%%%%%%%%%%%%%%%%%%%%%%%%%%%%%%%%%%%%%%%%%%%%%%%%%%%%%%%%%%%%%%%%%%%%%%%%%%

\section{Hyperelliptic Rauzy diagrams and Rauzy--Veech groups $RV(k,\ell)$}\label{s.RV}

\subsection{Review of the description of hyperelliptic Rauzy diagrams} 

In this subsection, we recall the content of Subsection 2.1 of our previous paper \cite{AMY-hyp}. 

Given an integer $d \geq 2$, denote by $\Ad$ the arithmetic progression $d-1, \, d-3, \ldots , 1-d$. Note that $\Ad$ has a natural involution $\iota(k)=-k$. 

The hyperelliptic Rauzy class $\Rd$ over $\Ad$ and the associated Rauzy diagram $\Dd$ are inductively defined as follows. 

The Rauzy class $\Rd$ contains the central vertex $\pi^* = \pi^*(d)$: 
$$\pi^*_t(k) = \frac 12 (d+1+k),\quad \pi^*_b(k) = \frac 12 (d+1-k).$$

For $d=2$, this is the sole vertex. For $d \geq 2$, $\R_{d+1}$ is the disjoint union of $\pi^*(d+1)$, $j_t(\Rd)$ and $j_b(\Rd)$, where $j_t$, $j_b$ are the following injective maps: for  $\pi \in \Rd$, if we denote $j_t(\pi) = t\pi$, $j_b(\pi) = b\pi$, then 

$$t\pi_t(-d) =1,\quad \quad \quad t\pi_b(-d) = \pi_b(d-3) ,$$
$$t\pi_t(k) = 1+\pi_t(k-1),$$
$$  t\pi_b(k) = \left \{ 
\begin{array}{cc}
 \pi_b(k-1)& \text{if } \pi_b(k-1) < \pi_b(d-3),\\
\pi_b(k-1)+1 & \text{if } \pi_b(k-1) \geq \pi_b(d-3),
\end{array} \right.$$
for $2-d \leq k \leq d$, and

$$b\pi_b(d) =1,\quad \quad \quad b\pi_t(d) = \pi_t(3-d) ,$$
$$b\pi_b(k) = 1+\pi_b(k+1),$$
$$  b\pi_t(k) = \left \{ 
\begin{array}{cc}
 \pi_t(k+1)& \text{if } \pi_t(k+1) < \pi_t(3-d),\\
\pi_t(k+1)+1 & \text{if } \pi_b(k+1) \geq \pi_t(3-d),
\end{array} \right.$$
for $-d \leq k \leq d-2$.

The arrows of $\Dd$ are given by the following one-to-one maps $R_t, R_b$ from $\Rd$ to itself: 

$$\left \{  \begin{array}{cc} R_t(\pi^*(d+1)) = j_t (\pi^*(d)),\\
 R_b(\pi^*(d+1)) = j_b (\pi^*(d)), \end{array} \right. $$
 
 \medskip
 
$$ \left \{  \begin{array}{cc} R_t \circ j_b \circ R_t^{-1} =j_b, \\
R_b \circ j_t \circ R_b^{-1} =j_t, \end{array} \right. $$

\medskip

$$ \left \{  \begin{array}{cc} R_t \circ j_t \circ R_t^{-1}(\pi) =j_t(\pi), \quad \quad \pi \ne \pi^*(d),\\
R_b \circ j_b \circ R_b^{-1}(\pi) =j_b(\pi), \quad \quad \pi \ne \pi^*(d),\end{array} \right.$$

\medskip

$$R_t \circ j_t \circ R_t^{-1}(\pi^*(d)) = \pi^*(d+1)= R_b \circ j_b \circ R_b^{-1}(\pi^*(d)).$$

The elements of $\Rd$ correspond bijectively to the words in $\{t,b\}$ of length $<d-1$ via the following map $W_d$: let $W_d(\pi^*(d))$ be the empty word, $W_d (j_t(\pi))$ is the word $tW_{d-1}(\pi)$ and $W_d (j_b(\pi))$ is the word $bW_{d-1}(\pi)$. 

One recovers from $W_d(\pi)$ the winners of the arrows starting from $\pi$ as follows: the winner of the arrow of top type starting from $\pi$ is the letter $d-1-2w_b(\pi)$ of $\Ad$, where $w_b(\pi)$ is the number of occurrences of $b$ in $W_d(\pi)$; similarly, the winner of the arrow of bottom type starting from $\pi$ is the letter $1-d+2w_t(\pi)$ of $\Ad$. Observe that we have always
$$ d-1-2w_b(\pi)> 1-d+2w_t(\pi).$$

A vertex $\pi \in \Rd$ is connected to the central vertex $\pi^*(d)$ by a \emph{unique} oriented \emph{simple} path $\gamma^*(\pi)$ in $\Dd$ from $\pi^*(d)$ to $\pi$. 

As it turns out, all non-trivial \emph{simple} loops in $\Rd$ are \emph{elementary}: they consist of arrows of the same type. Any such loop $\gamma$ contains a unique vertex $\pi$ such that $\gamma$ passes through $\pi$ but $\gamma^*(\pi)$ does not contain any arrow of $\gamma$, and, furthermore, $\pi$ is the vertex of $\gamma$ such that $ |W_d(\pi)|$ is minimal, and its value is $d-1-|\gamma|$. In the sequel, we denote by $\gamma'$ the non-oriented loop at $\pi^*(d)$ defined by 
$$\gamma' := \gamma^*(\pi) * \gamma * (\gamma^*(\pi))^{-1}.$$

\subsection{Mapping classes attached to the arrows of $\mathcal{D}_d$} 

In this subsection, we essentially review the content of Section 4 of our previous paper \cite{AMY-hyp}. 

Given $\pi \in \Rd$, denote by $M_{\pi}$ the \emph{canonical} translation surface with combinatorial data $\pi$ whose length data $\lambda^{can}$ and suspension data $\tau^{can}$ are: 
 $$ \lambda^{can}_\alpha = 1, \quad \tau^{can}_\alpha = \pi_b(\alpha) - \pi_t(\alpha), \quad \forall \alpha \in \A.$$

We obtain $M_{\pi}$ by identifying parallel sides of an appropriate polygon $P_{\pi}$. The set of marked points of $M_{\pi}$ is denoted by $\Sigma_{\pi}$, and the middle points of the sides of $P_{\pi}$ together with a base point $O_{\pi}$ in the interior of $P_{\pi}$ form a subset $\Sigma^*_{\pi}$ of $M_{\pi}$. 

The Rauzy--Veech operation associated to each arrow $\gamma: \pi \to \pi'$ of $\Dd$ is encoded by the isotopy class of a homeomorphism $H_\gamma : (M_\pi, \Sigma_\pi \cup \Sigma^*_\pi) \to (M_{\pi'}, \Sigma_{\pi'} \cup \Sigma^*_{\pi'})$ constructed in Subsection 4.1 of \cite{AMY-hyp}. 

The map $\gamma\mapsto[H_{\gamma}]$ induces (by functoriality\footnote{I.e., the image of a loop under the morphism is the composition of the $[H_{\gamma}]$ attached to its arrows.}) a morphism $\pi_1(\widetilde{\Dd}, \pi^*)\to \textrm{Mod}(\pi^*)$ from the fundamental group $\pi_1(\widetilde{\Dd}, \pi^*)$ of the non-oriented Rauzy diagram $\widetilde{\Dd}$ associated to $\Dd$ based at the central vertex $\pi^*$ to the mapping class group $\textrm{Mod}(\pi^*)$ of $(M_{\pi^*}, \Sigma_{\pi^*})$.  

In this context, given $\gamma$ a simple loop in $\Dd$, the action of $\gamma'$ as a isotopy class on $M_{\pi^*}$ was computed in Proposition 4.5 of \cite{AMY-hyp}: it is a Dehn twist about the straight line joining the midpoints of the sides of $P_{\pi^*}$ indexed by the letter of $\mathcal{A}_d$ winning in the loop $\gamma$. 

Note that the elements of $\textrm{Mod}(\pi^*)$ can be viewed also as mapping classes on the translation surface $\mathcal{M}_{1,\ell}$ where $\ell=d+1$, and, \emph{a fortiori}, they can be lifted to $\mkl$ via the natural projection $\mkl\to\mathcal{M}_{1,\ell}$. 

\begin{definition}\label{d.RV} The Rauzy--Veech group $RV(k,\ell)$ is the group generated by the actions on $H_1(\mkl, \mathbb{R})$ of all $\gamma'$ associated to all elementary loops $\gamma$ in $\Rd$ (where $\ell=d+1$). 
\end{definition}

The Rauzy--Veech group $RV(1,\ell)$ was computed in our previous paper \cite{AMY-hyp}: it is isomorphic to an explicit finite-index subgroup of $Sp(2g,\mathbb{Z})$ (where $g$ is the genus of $\mathcal{M}_{1,\ell}$). 

\begin{remark}\label{r.RV-Hk} In general, the natural projection $\mkl\to \mathcal{M}_{1,\ell}$ takes $H_k$ to $H_1(\mathcal{M}_{1,\ell},\mathbb{R})$ in such a way that $RV(k,\ell)|_{H_k}$ is isomorphic to $RV(1,\ell)$, so that $RV(k,\ell)|_{H_k}$ is the explicit finite-index subgroup described in Theorem 2.9 of our previous paper \cite{AMY-hyp}. 
\end{remark}

\subsection{Lifting the action of the loops in $\Dd$: top case}\label{ss.Lt}

Let $\gamma$, $\pi$, $\gamma'$ be as in the previous subsection. Let $k$ be an integer $\geq 2$. Let $\mathcal M_{k,d+1}$
be the surface considered in the first section. We have a canonical projection
$\mathcal M_{k,d+1} \rightarrow \mathcal M_{1,d+1}.$
The action of $\gamma'$, as an isotopy class of the translation surface $\mathcal M_{1,d+1}$ with marked points at the $M(j)$ and the $A(\delta)$, was already discussed in the previous subsection. We describe now the lift of this action to $\mathcal M_{k,d+1}$. 

Assume first that $\gamma$ is of top type. Let $w:=w_b(\pi) \in \{0,\ldots,d-2 \}$. The winner\footnote{In the sequel, it is useful to recall that an elementary loop of given (top or bottom) type is uniquely determined the parameter $w$, or equivalently, the parameter $p$.} of $\pi$ is then $p:=d-1-2w$. We write $L_p^t$ for the action of $\gamma'$ on the homology of $\mathcal M_{k,d+1}$. 

For $i \in \zk$, $j \in \Zset_{d+1}$, the image of the relative homology class $\aleph ((i,i+1),(j,j+1))$ by $L_p^t$ is equal to

\begin{itemize}
\item  $\aleph ((i,i+1),(j,j+1))$, if $j$ is neither $0$ nor $w+1$;

\medskip

\item $\aleph((i-1,i),(0,1)) + \sum_1^w (\aleph((i-1,i),(m,m+1)) - \aleph((i,i+1),(m,m+1)))$, if $j=w+1$;

\medskip

\item \begin{eqnarray*}\aleph((i-1,i),(w+1,w+2)) + \sum_{w+2}^d (\aleph((i-1,i),(m,m+1)) \\ -  \aleph((i,i+1),(m,m+1)))\end{eqnarray*} if $j=0$.
\end{itemize}

We define
$$V_i(p):= \sum_0^w \aleph((i-1,i),(m,m+1)) - \sum_1^{w+1}  \aleph((i,i+1),(m,m+1)),$$
so that we have
$$L_p^t(\aleph((i,i+1),(j,j+1))) = \aleph((i,i+1),(j,j+1)) + (\delta_{j,w+1} - \delta_{j,0})V_i(p).$$

This gives
$$L_p^t(V_i(p)) = - V_{i-1}(p).$$

\medskip

Let $0<r<k, 0<s<\ell := d+1$ with $\frac r{2k} + \frac s{\ell} \ne 1$; write
$$x_{i,j} = \cos 2\pi(\frac {ri}{2k} + \frac {sj}{\ell} ), \quad y_{i,j} = \sin 2\pi(\frac {ri}{2k} + \frac {sj}{\ell} ).$$

Define
$$X^t_{r,s} := \sum_{i \in \zk,\, j \in \zl} x_{i,j} \aleph((i,i+1),(j,j+1)), \quad Y^t_{r,s} := \sum_{i \in \zk,\, j \in \zl} y_{i,j} \aleph((i,i+1),(j,j+1)).$$

We have 
$$L_p^t(X^t_{r,s}) = X^t_{r,s} + \sum_{i \in \zk} (x_{i,w+1} - x_{i,0}) V_i(p),$$
$$L_p^t(Y^t_{r,s}) = Y^t_{r,s} + \sum_{i \in \zk} (y_{i,w+1} - y_{i,0}) V_i(p).$$

Set 
$$V_{cos}(p,r):= \sum_{i \in \zk} x_{i,0} V_i(p), \quad V_{sin}(p,r):= \sum_{i \in \zk} y_{i,0} V_i(p).$$

Then we have 
\begin{eqnarray*}
L_p^t(V_{cos}(p,r)) &= & \sum_{i \in \zk} x_{i,0} L_p^t(V_i(p)) \\
                &=& - \sum_{i \in \zk} x_{i,0} V_{i-1}(p) \\
                &=& - \sum_{i \in \zk} x_{i+1,0} V_{i}(p) \\
                &=& \cos (\pi (1 + \frac rk )) V_{cos}(p,r) - \sin (\pi (1 + \frac rk )) V_{sin}(p,r)
\end{eqnarray*}
and
$$L_p^t(V_{sin}(p,r)) = \sin (\pi (1 + \frac rk )) V_{cos}(p,r) + \cos (\pi (1 + \frac rk )) V_{sin}(p,r).$$

We also compute the image of $V_{cos}(p',r)$ and $V_{sin}(p',r)$ for $p' \ne p$. We write  $p'= d-1-2w'$. Assume first that $p'< p$ (i.e., $w'>w$). One has

\begin{eqnarray*} 
L_p^t(V_{cos}(p',r)) &= & \sum_{i \in \zk} x_{i,0}\left(\sum_0^{w'} L_p^t(\aleph((i-1,i),(m,m+1))) - \sum_1^{w'+1} L_p^t(\aleph((i,i+1),(m,m+1)))\right)\\
   &= & V_{cos}(p',r) + \sum_{i \in \zk}x_{i,0}(\sum_0^{w'} (\delta_{m,w+1} - \delta_{m,0})V_{i-1}(p) - \sum_1^{w'+1} (\delta_{m,w+1} - \delta_{m,0})V_{i}(p))\\
   &=& V_{cos}(p',r) - \sum_{i \in \zk}x_{i,0} V_{i}(p)) \\
  &=& V_{cos}(p',r) - V_{cos}(p,r).
\end{eqnarray*}

Similarly
$$ L_p^t(V_{sin}(p',r))= V_{sin}(p',r) - V_{sin}(p,r).$$

\bigskip 

When $p'> p$ (i.e., $w'<w$), one obtains
\begin{eqnarray*} 
L_p^t(V_{cos}(p',r)) &= & V_{cos}(p',r) + \sum_{i \in \zk}x_{i,0}(\sum_0^{w'} (\delta_{m,w+1} - \delta_{m,0})V_{i-1}(p) - \sum_1^{w'+1} (\delta_{m,w+1} - \delta_{m,0})V_{i}(p))\\
   &=& V_{cos}(p',r) - \sum_{i \in \zk}x_{i,0} V_{i-1}(p)\\
    &=& V_{cos}(p',r) + L_p^t( V_{cos}(p,r)),
\end{eqnarray*}
and similarly
$$ L_p^t(V_{sin}(p',r))= V_{sin}(p',r) + L_p^t( V_{sin}(p,r)).$$

In summary, for each $0<r<k$, we deduce that 
\begin{itemize}
\item The subspace $H_r$ is fixed by the lift of the action of $\gamma'$ (which commutes with the action of $\mathbb{Z}_{2k}$); therefore, the Rauzy--Veech group $RV(k,\ell)$ gives rise to well-defined  groups $RV(k,\ell)|_{H_r}$ (obtained by restriction to $H_r$); 
\item $H_r$ is the direct sum of the $2$-dimensional subspace generated by $V_{\cos}(p,r)$ and $V_{\sin}(p,r)$ on which $\gamma'$ acts by a rotation of $-\pi(1+\frac{r}{k})$, and a subspace of codimension $2$ on which $\gamma'$ acts by the identity. 
\end{itemize}

\subsection{Lifting the action of the loops in $\Dd$: bottom case}

In the same setting that in the last subsection, we now assume that $\gamma$ is of bottom type.

 Let $w:=w_t(\pi) \in \{0,\ldots,d-2 \}$. The winner of $\pi$ is then $p:=-d+1+2w$. We write $L_p^b$ for the action of $\gamma'$ on the homology of $\mathcal M_{k,d+1}$. 

For $i \in \zk$, $j \in \Zset_{d+1}$, the image of $\aleph((i-1,i),(-j,-j+1))$ by $L_p^b$ is equal to

\begin{itemize}
\item $\aleph((i-1,i),(-j,-j+1))$, if $j$ is neither $0$ nor $w+1$;
\item \begin{eqnarray*}
\aleph((i,i+1),(0,1))+ \sum_1^w (\aleph((i,i+1),(-m,-m+1)) \\ - \aleph((i-1,i),(-m,-m+1))),
\end{eqnarray*} if $j=w+1$;
\item \begin{eqnarray*}
\aleph((i,i+1),(-w-1,-w)) + \sum_{w+2}^d (\aleph((i,i+1),(-m,-m+1)) \\ - \aleph((i-1,i),(-m,-m+1))),
\end{eqnarray*} if $j=0$.
\end{itemize}

We now have 
$$ \sum_0^w \aleph((i,i+1),(-m,-m+1)) - \sum_1^{w+1} \aleph((i-1,i),(-m,-m+1)) = V_i(p)$$

hence 
$$L_p^b(\aleph((i-1,i),(-j,-j+1))) = \aleph((i-1,i),(-j,-j+1)) + (\delta_{j,w+1} - \delta_{j,0})V_i(p).$$

This gives
$$L_p^b(V_i(p)) = - V_{i+1}(p).$$

\bigskip

Let $0<r<k, 0<s<\ell$ with $\frac r{2k} + \frac s{\ell} \ne 1$;
% write
%
%$$x_{i,j} = \cos 2\pi(\frac {ri}{2k} + \frac {sj}{\ell} ), y_{i,j} = \sin 2\pi(\frac {ri}{2k} + \frac {sj}{\ell} ).$$
%Define
%$$X^b_{r,s} := \sum_{i \in \zk,\, j \in \zl} x_{i,j} \gamma (\{i-1,i\},j),$$
%$$Y^b_{r,s} := \sum_{i \in \zk,\, j \in \zl} y_{i,j} \gamma (\{i-1,i\},j).$$
%We have 
%$$L_p^b(X^b_{r,s}) = X^b_{r,s} + \sum_{i \in \zk} (x_{i,-w-1} - x_{i,0}) V_i(p),$$
%$$L_p^b(Y^b_{r,s}) = Y^b_{r,s} + \sum_{i \in \zk} (y_{i,-w-1} - y_{i,0}) V_i(p).$$
we have 
\begin{eqnarray*}
L_p^b(V_{cos}(p,r)) &= & \sum_{i \in \zk} x_{i,0} L_p^b(V_i(p)) \\
                &=& - \sum_{i \in \zk} x_{i,0} V_{i+1}(p) \\
                &=& - \sum_{i \in \zk} x_{i-1,0} V_{i}(p) \\
                &=& \cos (\pi (1 + \frac rk )) V_{cos}(p,r) + \sin (\pi (1 + \frac rk )) V_{sin}(p,r)
\end{eqnarray*}
and 
$$L_p^b(V_{sin}(p,r)) = -\sin (\pi (1 + \frac rk )) V_{cos}(p,r) + \cos (\pi (1 + \frac rk )) V_{sin}(p,r).$$

We also compute the image of $V_{cos}(p',r)$ and $V_{sin}(p',r)$ for $p' \ne p$. We write  $p'= -d+1+2w'$. Assume first that $p'> p$ (i.e., $w'>w$). One has 
\begin{eqnarray*} 
L_p^b(V_{cos}(p',r)) &= & \sum_{i \in \zk} x_{i,0}(\sum_0^{w'} L_p^b(\aleph((i,i+1),(-m,-m+1)) - \sum_1^{w'+1} L_p^b(\aleph((i-1,i),(m,m+1)))\\
   &= & V_{cos}(p',r) + \sum_{i \in \zk}x_{i,0}(\sum_0^{w'} (\delta_{m,w+1} - \delta_{m,0})V_{i+1}(p) - \sum_1^{w'+1} (\delta_{m,w+1} - \delta_{m,0})V_{i}(p))\\
   &=& V_{cos}(p',r) - \sum_{i \in \zk}x_{i,0} V_{i}(p) \\
  &=& V_{cos}(p',r) - V_{cos}(p,r).
\end{eqnarray*}

Similarly
$$ L_p^b(V_{sin}(p',r))= V_{sin}(p',r) - V_{sin}(p,r).$$

When $p'< p$ (i.e., $w'<w$), one obtains
\begin{eqnarray*} 
L_p^b(V_{cos}(p',r)) &= & V_{cos}(p',r) + \sum_{i \in \zk}x_{i,0}(\sum_0^{w'} (\delta_{m,w+1} - \delta_{m,0})V_{i+1}(p) - \sum_1^{w'+1} (\delta_{m,w+1} - \delta_{m,0})V_{i}(p))\\
   &=& V_{cos}(p',r) - \sum_{i \in \zk}x_{i,0} V_{i+1}(p)\\
    &=& V_{cos}(p',r) + L_p^b( V_{cos}(p,r)),
\end{eqnarray*}
and similarly
$$ L_p^b(V_{sin}(p',r))= V_{sin}(p',r) + L_p^b( V_{sin}(p,r)).$$

Thus we see that $L_p^b$ is the inverse of $L^t_p$.  

In summary, the group $RV(k,\ell)|_{H_r}$ is generated by the operators $L_p^t|_{H_r}$. 

\subsection{Formulas for $L_p^t$ as a complex operator}

Let $\rho:=\exp(i\pi\frac rk)$. The complex structure on $H_r$ is given by $1_{2k} .v = \rho v$. We have thus

 \begin{eqnarray*}
L_p^t(V_{cos}(p,r)) &= & \sum_{i \in \zk} x_{i,0} L_p^t(V_i(p)) \\
                &=& - \sum_{i \in \zk} x_{i,0} V_{i-1}(p) \\
                &=& -\rho^{-1} V_{cos}(p,r).
    \end{eqnarray*}            

For $p'>p$,
$$ L_p^t(V_{cos}(p',r)) = V_{cos}(p',r) - \rho^{-1} V_{cos}(p,r),$$

and for $p'<p$
$$  L_p^t(V_{cos}(p',r)) = V_{cos}(p',r) - V_{cos}(p,r).$$

\subsection{Computation of the hermitian form on $H_r$}\label{ss.signature}

We fix $0<r<k$ and abbreviate $Z(p):= V_{cos}(p,r)$, $x_i:=x_{i,0}$. We first compute the hermitian product

$$\langle Z(p),Z(p)\rangle:= (\sin(\pi \frac rk))^{-1} \omega(1_{2k} .Z(p),Z(p)).$$

\begin{lemma}\label{l.intersections}
Let $(a_i)_{i \in \zk}$ and $(b_i)_{i \in \zk}$ be real numbers with $\sum_i a_i = \sum_i b_i =0$. Then, we have 
$$\omega( \sum_i a_i V_i(p), \sum_i b_i V_i(p)) = \sum _{1 \leq i \leq i' <2k} (a_i b_{i'} - a_{i'} b_i).$$
For $p>p'$, we have
$$\omega( \sum_i a_i V_i(p), \sum_i b_i V_i(p')) = \sum _{1 \leq i \leq i' <2k} a_i b_{i'} .$$
\end{lemma}

\begin{proof}
Observe first that, although $V_i(p)$ is only a relative homology class with nonzero boundary equal to $M(0,1) - M(w+1,w+2)$ (where $p=d-1-2w$), the condition $\sum_i a_i =0$ insures that $ \sum_i a_i V_i(p)$ is an absolute homology class so the intersection form $\omega$ is well-defined in the formulas of the lemma.

\medskip

In the second case, it is easy to represent the cycles $V_i(p)$ (resp. $V_i(p')$) by paths from $M(w+1,w+2)$ (resp. $M(w'+1,w'+2)$) to $M(0,1)$ in $Q_i$ so that the intersection takes place only at $M(0,1)$. A direct inspection at this point gives the formula of the lemma.

\medskip

In the first case, we choose two distinct representations for each $V_i(p)$ as paths from $M(w+1, w+2)$ to $M(0,1)$ so that the intersection takes place only at $M(w+1,w+2)$ and $M(0,1)$. Again a direct inspection at these points gives the formula of the lemma.
\end{proof}

Using the first part of the lemma, we get (as $ \sum_i x_i =0$) 
\begin{eqnarray*}
 \omega(1_{2k} .Z(p),Z(p)) &=& \omega ( \sum x_i 1_{2k}.V_i(p), \sum x_i V_i(p)) \\
                          &=&  \omega ( \sum x_i V_{i+1}(p),  \sum x_i V_i(p)) \\
                          &=& \omega ( \sum x_{i-1} V_i(p),  \sum x_i V_i(p)) \\
                          &=& \sum _{1 \leq i \leq i' <2k} (x_{i-1}x_{i'} - x_{i}x_{i'-1})\\
                          &=& \sum_1^{2k-1} x_0 x_{i'} + \sum_1^{2k-2} x_i x_{2k-1} - \sum_1^{2k-1} x_i x_{i-1} - \sum_1^{2k-2} x^2_i \\
                          &=& -x_0^2 -x_{-1}(x_0 + x_{-1}) - \sum_1^{2k-1} x_i x_{i-1} - \sum_1^{2k-2} x^2_i \\
                          &=& -\sum_{i \in \zk} (x_i x_{i-1}+x^2_i) \\
                          &=& - \frac 14 \sum_{i \in \zk} (\rho ^i + \rho ^{-i})(\rho ^i + \rho ^{-i}+\rho ^{i-1} + \rho ^{-i+1})\\
                          &=&  - \frac k2 (2 + \rho + \rho^{-1}) \\
                           &=& -k(1+ \cos \pi \frac rk).
\end{eqnarray*}

We have thus 
$$ \langle Z(p),Z(p)\rangle:= -k \frac {1+ \cos \pi \frac rk}{\sin(\pi \frac rk)}.$$

We now compute 
 $$ 2\Re (\langle Z(p),Z(p')\rangle) = \langle Z(p)+Z(p'), Z(p)+Z(p')\rangle - \langle Z(p),Z(p)\rangle -  \langle Z(p'),Z(p')\rangle$$
for $p>p'$. We have (using the second part of Lemma \ref{l.intersections}) 

\begin{eqnarray*}
  2 \sin(\pi \frac rk) \Re (\langle Z(p),Z(p')\rangle) &=& \omega(1_{2k} .Z(p),Z(p')) + \omega(1_{2k} .Z(p'),Z(p)) \\
          &=& \omega (\sum x_{i-1} V_i(p),  \sum x_i V_i(p')) - \omega ( \sum x_i V_i(p), \sum x_{i-1} V_i(p')) \\
          &=& \sum _{1 \leq i \leq i' <2k} (x_{i-1}x_{i'} - x_{i}x_{i'-1})\\
          &=& -k(1+ \cos \pi \frac rk).
\end{eqnarray*}

This gives 

$$ \Re (\langle Z(p),Z(p')\rangle)= -\frac k2 \frac {1+ \cos \pi \frac rk}{\sin(\pi \frac rk)}.$$

We finally compute 
$$ 2\Re (\rho \langle Z(p),Z(p')\rangle) = \langle \rho Z(p)+ Z(p')), \rho Z(p)+  Z(p')\rangle - \langle Z(p),Z(p)\rangle -  \langle Z(p'),Z(p')\rangle$$
for $p>p'$. We have 
\begin{eqnarray*}
  2 \sin(\pi \frac rk) \Re ( \rho \langle Z(p),Z(p')\rangle) &=& \omega(1_{2k}^2 .Z(p), Z(p')) + \omega(1_{2k} .Z(p'),1_{2k}Z(p)) \\
          &=&   \omega(1_{2k} .Z(p), 1_{2k}^{-1} .Z(p')) + \omega(Z(p'),Z(p)) \\
          &=& \omega (\sum x_{i-1} V_i(p),  \sum x_{i+1} V_i(p')) - \omega ( \sum x_i V_i(p), \sum x_{i} V_i(p') )\\
          &=& \sum _{1 \leq i \leq i' <2k} (x_{i-1}x_{i'+1} - x_{i}x_{i'})  \\
          &=& \sum_1^{2k-1} x_0 x_{i'+1} + \sum_2^{2k-1} x_{i-1} x_{2k} - \sum_1^{2k-2} x_i x_{i+1} - \sum_1^{2k-1} x^2_i \\
          &=& -\sum_{i \in \zk} (x_i x_{i+1}+x^2_i) \\
          &=& -k(1+ \cos \pi \frac rk).
\end{eqnarray*}

This gives 

$$ \Re (\rho \langle Z(p),Z(p')\rangle)= -\frac k2 \frac {1+ \cos \pi \frac rk}{\sin(\pi \frac rk)}.$$

We finally get, for $p>p'$, 
\begin{eqnarray*}
\Im ( \langle Z(p),Z(p')\rangle)&=& (\Im(\rho))^{-1} ( \Re(\rho)\Re ( \langle Z(p),Z(p')\rangle) -  \Re (\rho \langle Z(p),Z(p')\rangle)) \\
    &=& -\frac k2 \frac {1+ \cos \pi \frac rk}{\sin(\pi \frac rk)} \frac {-1+ \cos \pi \frac rk}{\sin(\pi \frac rk)}\\
    &=& \frac k2.
\end{eqnarray*}

It is probably nicer to scale the hermitian form by the factor $-k\frac {1+ \cos \pi \frac rk}{\sin(\pi \frac rk)}$ in order to have $ \langle Z(p),Z(p)\rangle =1$ for all $p$. One has then, for $p>p'$
$$  \langle Z(p),Z(p')\rangle = \frac 12 ( 1 -i \tan \frac {\pi r}{2k}).$$

Let $u:=  \tan \frac {\pi r}{2k}$. Consider the hermitian form on $\Cset^{\Ad}$ defined by
$$ A((w(p))_{p\in\Ad}):= \langle \sum w(p)Z(p), \sum w(p) Z(p)\rangle = A_1 - u A_2,$$
with
$$ A_1 = \frac 12 ( \sum_p |w(p)|^2 + |\sum_p w(p) |^2)$$
and 
$$ A_2 = \sum_{p>p'} \Im (w(p') \overline {w(p)} ).$$

Observe that $A_1$ is positive, so we can diagonalize simultaneously $A_1$ and $A_2$. Let indeed $\xi:= \exp \frac {2i\pi}{\ell}$ (recall that $\ell = d+1$). Define, for $0<s<\ell$, $p=d-1-2w$

$$w_s(p) = \xi^{sw}.$$

We have then, for $z_1, \ldots, z_d \in \Cset$

\begin{eqnarray*}
A_1(\sum_s z_s w_s) &=& \frac 12 ( \sum_p |\sum_s z_s \xi^{s w}|^2 + |\sum_p \sum_s z_s  \xi^{s w} |^2 )\\
        &=& \frac 12 ( \sum_p \sum_s \sum_{s'} z_s \overline{z_{s'}} \xi^{(s-s')w} + |\sum z_s \xi ^{-s}|^2) \\
        &=& \frac { \ell}2 \sum_s |z_s|^2,
\end{eqnarray*}        
and
\begin{eqnarray*}
A_2(\sum_s z_s w_s) &=& \sum_{p>p'} \Im (\sum_{s'}\sum_s z_{s'} \overline {z_s} w_{s'}(p') \overline {w_s(p)}) \\
                 &=& \sum_{s'}\sum_s \Im ( z_{s'} \overline {z_s}\sum_{0 \leq w <w'<\ell -1} \xi^{s'w'-s w}).
\end{eqnarray*}  
For $s\ne s'$, one has
\begin{eqnarray*}
\sum_{0 \leq w <w'<\ell -1} \xi^{s'w'-s w} &=& \sum_{0<w'<\ell -1} \xi^{s'w'} \frac {\xi^{-sw'} -1}{\xi^{-s} -1} \\
                                   &=& \frac {-1-\xi^{s-s'}+1+\xi^{-s'}} {\xi^{-s} -1} \\
                                   &=& \xi^{s-s'},
\end{eqnarray*}  
and 
$$\Im ( z_{s'} \overline {z_s} \xi^{s-s'} + z_{s} \overline {z_s'} \xi^{s'-s}) =0.$$
On the other hand, for $0<s<\ell$, we have
\begin{eqnarray*}
\sum_{0 \leq w <w'<\ell -1} \xi^{s w'-s w} &=& \sum_{0<w'<\ell -1} \xi^{s w'} \frac {\xi^{-sw'} -1}{\xi^{-s} -1} \\
      &=& \frac {\ell + \xi^{-s} -1}{\xi^{-s} -1}
\end{eqnarray*}
We have therefore
\begin{eqnarray*}
A_2(\sum_s z_s w_s) &=& \sum_s \Im  (|z_s|^2 \sum_{0 \leq w <w'<\ell -1} \xi^{s w'-s w})\\
                    &=& \ell  \sum_s |z_s|^2 \Im \frac 1{\xi^{-s} -1} \\
                    &=& \frac { \ell}2 \sum_s (\tan \frac{\pi s}{\ell})^{-1} |z_s|^2.
\end{eqnarray*}

In particular, this shows that the hermitian form on $H_r$ has the signature described in Theorem \ref{t.A} (as expected from McMullen's paper \cite{McM}). 

\subsection{Matricial description of $RV(k,\ell)|_{H_r}$}\label{ss.reduction} Our discussion so far is summarized as follows. The group $RV(k,\ell)|_{H_r}$ is generated by the operators $L_p^t$, $p\in\Ad$, given by  
$$  L_p^t(V_{\cos}(p',r)) = \left \{ \begin{array}{cc}  
V_{\cos}(p',r) - \rho^{-1} V_{\cos}(p,r)   & \text{if} \; p'>p  \\ 
V_{\cos}(p',r) - V_{\cos}(p,r)   &  \text{if} \; p'<p \\
 -\rho^{-1} V_{\cos}(p,r)   &  \text{if} \; p'=p
  \end{array} \right.  $$ 
These operators preserve the hermitian form 
$$\langle V_{\cos}(p,r), V_{\cos}(p,r)\rangle =1, \quad \quad \langle V_{\cos}(p,r), V_{\cos}(p',r)\rangle = \frac 12 ( 1 -i \tan \frac {\pi r}{2k}) \quad \forall\, p>p'$$ 
  
At this point, we reduced the proof of Theorem \ref{t.A} to the analysis of the group generated by the matrices above.

%%%%%%%%%%%%%%%%%%%%%%%%%%%%%%%%%%%%%%%%%%%%%%%%%%%%%%%%%%%%%%%%%%%%%%%%%%%%%%%%

\section{Matrices}\label{s.matrices}

In this section, we complete the proof of Theorem \ref{t.A} by studying the groups of matrices. Since we do not need anymore to make reference to the homology groups translation surfaces $\mkl$, we are going to rewrite below the formulas from Subsection \ref{ss.reduction} using a slightly more abstract notation. 

\subsection{Setting}

\begin{itemize}
\item $d$ is an integer $\ge 2$, $\ell = d+1$.
\item $\rho$ is a complex number of modulus $1$ with positive imaginary part, frequently a root of unity; $\zeta:= \rho^{-1}$. We write $\rho = \exp 2 \pi i \alpha$, $\alpha \in (0,\frac 12)$.
\item $p$ is an integer running from $1-d$ to $d-1$ with step $2$, and thus taking $d$ values. We denote by $\Ad$ the set of values of $p$.
\item $(e_p)$ is the canonical basis of $\Cset^{\Ad}$.
\end{itemize}

\subsection{The operators}\label{operators}

For $p \in \Ad$, we define an operator $L_p =L_p^t$ on  $\Cset^{\Ad}$ by 
$$  L_p(e_q) = \left \{ \begin{array}{cc}  
e_q - e_p   & \text{if} \; q>p  \\ 
e_q - \zeta e_p   &  \text{if} \; q<p \\
 -\zeta e_p   &  \text{if} \; q=p
  \end{array} \right.  $$

Observe that the inverse $L_p^{-1} =L_p^b$ is given by 
$$  L_p^b(e_q) = \left \{ \begin{array}{cc}  
e_q - e_p   & \text{if} \; q<p  \\ 
e_q  -  \rho e_p   &  \text{if} \; q>p \\
 -\rho e_p   &  \text{if} \; q=p
  \end{array} \right.  $$
i.e. the same formula, changing $\zeta$ to $\rho$ and inverting the order on $\Ad$.

\subsection{The invariant hermitian form}\label{ssdiag}

Let $Q_{\alpha}$ be the hermitian form on $\CAD$  such that
$$ Q_{\alpha}(e_p) = 1,  \quad \forall p \in \Ad,  \quad Q_{\alpha}(e_p,e_{p'}) = (1 + \zeta)^{-1} = \frac 12 (1 + i \tan \pi \alpha ) , \forall p > p' .$$

\subsection{The case $d=2$}\label{ssd2}

\smallskip

\subsubsection{A special element}

\smallskip

We compute $L_{-1}^b \circ L_1^t$

$$\left\{  \begin{array}{llll}
L_{-1}^b \circ L_1^t(e_{-1}) &= L_{-1}^b (e_{-1} -\zeta e_1) &= -\rho e_{-1} -\zeta(e_1 - \rho e_{-1})& = (1-\rho) e_{-1} -\zeta e_1,\\
L_{-1}^b \circ L_1^t(e_{1})&= L_{-1}^b(-\zeta e_1) &= -\zeta( e_1 -\rho e_{-1}) &= e_{-1} - \zeta e_1. 
\end{array} \right.$$

\medskip

We thus have $\det (L_{-1}^b \circ L_1^t) =1$ and $\textrm{tr} (L_{-1}^b \circ L_1^t) = 1 - ( \rho + \rho^{-1})$.

\begin{lemma}
The operator  $L_{-1}^b \circ L_1^t$ has infinite order if $\rho$ is a root of unity but $\rho^4 \ne 1$, $\rho^6 \ne 1$. It is hyperbolic if $ \rho + \rho^{-1}<-1$.
\end{lemma}

\begin{proof}
The second assertion is clear. For the first, if $\rho$ is a root of unity but $\rho^4 \ne 1$, $\rho^6 \ne 1$, there exists a Galois-conjugate $\rho'$ of $\rho$ such that 
$ \rho' + \rho'^{-1}<-1$. Then the corresponding Galois-conjugate of the matrix of $L_{-1}^b \circ L_1^t$ has infinite order. The same is true of  the matrix of $L_{-1}^b \circ L_1^t$.
\end{proof}

\subsubsection{The cases $\alpha = \frac 16 , \frac 14$}

\smallskip

In these cases, the hermitian form has signature $(2,0)$. The coefficients of the matrices of the group generated by $L_1$ and $L_{-1}$ belong to the ring of integers of the quadratic field 
$\Qset(\rho)$ and are bounded, hence the group  generated by $L_1$ and $L_{-1}$ is finite.

\subsubsection{The case $\alpha = \frac 13$}

\smallskip

Let $\rho = j:= \exp \frac{2\pi i}3 $, $e:= e_{-1} +j e_1$, $f:=  e_{-1} +j^2 e_1$. We have $L_{-1}(e) = L_1(e) =e$ and $L_{-1}(f) = -j^2 e -j^2 f$, $L_1(f) =  -j e -j^2 f$. The subgroup generated by $L_{-1}$ and $L_1$ is therefore contained in 
$$ \Gamma := \{ L \in GL(2,\Cset), L(e) =e, L(f) = \mu f + \omega e, \mu^6 =1, \omega \in \Zset \oplus \Zset j \}.$$ 
Observe that  $L_{-1}^b \circ L_1^t$ is parabolic, i.e satisfies $\mu =1$, $\omega \ne 0$.

\subsubsection{The case $0 < \alpha < \frac 13$ ,$\alpha \ne \frac 14 , \frac 16$}

\smallskip

In this case, the hermitian form has signature $(2,0)$. Denote by $U(Q_{\alpha})$ and $SU(Q_{\alpha})$ the associated unitary and special unitary groups.

\smallskip

 The operator $L_{-1}^b \circ L_1^t$ belongs to $SU(Q_{\alpha})$ and has infinite order hence the closed (for the usual topology) subgroup generated by $L_{-1}^b \circ L_1^t$ is a one-parameter group isomorphic to a circle, consisting of those transformations of $SU(Q_{\alpha})$ having the same eigenvectors than $L_{-1}^b \circ L_1^t$. An infinitesimal generator of this one-parameter group is the element $X$ of the Lie algebra  $su(Q_{\alpha})$ which satisfies $X.v_{\pm} =\pm  i v_{\pm}$, where $v_{\pm}$ are the eigenvectors of $L_{-1}^b \circ L_1^t$.
 
 \smallskip

The eigenvalues $\lambda_{\pm}$ of $L_{-1}^b \circ L_1^t$ are solutions of $\lambda^2  - (1-\zeta -\rho) \lambda +1 =0$, and the corresponding eigenvectors are 
$$ v_{\pm} = a_{\pm} e_{-1} + e_1 ,\quad \quad  a_{\pm} = -1-\rho \lambda_{\pm}. $$ 
The coefficients of $X$ are given by
$$  X_{-1\,-1} = \frac {i (a_+ +a_-)}{a_+  -a_- }= -X_{1\,1} ,\quad X_{1\,-1} = \frac{2 i}{a_+  -a_- },\quad   X_{-1\,1} = \frac {2ia_+ a_-}{a_+  -a_- }.$$

For $n \in \Zset$,  write $\textrm{ad}(L_{-1}^n) X =:X(n)$, which is the infinitesimal generator of the previous one-parameter group conjugated by $L_{-1}^n$. Setting also 
$L_{-1}^n v_{\pm} =: v_{\pm}(n) =: a_{\pm}(n) e_{-1} + e_1$, we have 

$$  X_{-1\,-1}(n) = \frac {i (a_+(n) +a_-(n))}{a_+(n)  -a_-(n) }= -X_{1\,1}(n) ,\quad X_{1\,-1}(n) = \frac{2 i}{a_+(n)  -a_-(n) },$$
$$  X_{-1\,1}(n) = \frac {2ia_+(n) a_-(n)}{a_+(n)  -a_-(n) }$$
with the recurrence relation $a(n+1) = - \zeta a(n) -1$. 

\smallskip

We now show that the vector fields $X=X(0), X(1), X(2)$ are linearly independent (over $\Cset$). The sequences 
$s(n):= a_+(n) +a_-(n)$ and $p(n):= a_+(n) a_-(n)$ satisfy the recurrence relations $s(n+1) = -\zeta s(n) -2$, $p(n+1) = \zeta^2 p(n) + \zeta s(n) +1$ with initial conditions
$ s(0) = -1-\rho +\rho^2$, $p(0) = \rho$. We have thus

$$ s(1) = \zeta -1-\rho,\quad s(2) = -\zeta^2 +\zeta -1 ,\quad p(1)= \rho, \quad p(2) = \zeta^2.$$

As we have 

$$ \det \left ( \begin{array}{ccc}1&p(0) & s(0) \\ 1&p(1) & s(1) \\ 1&p(2) & s(2) \end{array} \right  ) = (1-\zeta^3)(1-\rho^3) \ne 0, $$

the vector fields $X=X(0), X(1), X(2)$ are indeed linearly independent. The Lie algebra $su(Q_{\alpha}) $ has dimension $3$, hence is generated by  $X(0), X(1), X(2)$ . 

\smallskip

We conclude that the intersection of $SU(Q_{\alpha})$ with the group generated by $L_1$ and $L_{-1}$ is dense (for the usual topology) in  $SU(Q_{\alpha})$.

\subsubsection{The case $\frac 13 < \alpha < \frac 12$}

In this case, the hermitian form has signature $(1,1)$. Denote by $U(Q_{\alpha})$ and $SU(Q_{\alpha})$ the associated unitary and special unitary groups.

\smallskip

 The operator $L_{-1}^b \circ L_1^t$ belongs to $SU(Q_{\alpha})$ and has eigenvalues $\lambda_+ >1$ and $\lambda_- = \lambda_+^{-1}$.
  Let $v_{\pm}$ be the associated eigenvectors. The Zariski closure of the 
 group generated by  $L_{-1}^b \circ L_1^t$  is the one-parameter group having for infinitesimal generator  a  vectorfield $X$ satisfying $X.v_{\pm} =\pm   v_{\pm}$. 
 The same calculation as in the previous case shows that $X$, $\textrm{ad}(L_{-1}) X$ and $\textrm{ad}(L_{-1}^2) X$ span $su(Q_{\alpha}) $. We conclude 
  that the intersection of $SU(Q_{\alpha})$ with the group generated by $L_1$ and $L_{-1}$ is dense (for the Zariski topology) in  $SU(Q_{\alpha})$.

\subsection{The induction step in the generic case}

\subsubsection{ Restrictions}

For $d \ge 3$, $p \in \Ad$, let $\iota_p$ be the embedding of $\A_{d-1}$ into $\Ad$ sending $q$ to $q-1$ if $q<p$ and to $q+1$ if $q>p$ (observe that $p \notin \Ad$).
The image of $\iota_p$ is $\Ad - \{p\}$. We also denote by $\iota_p$ the embedding of $\Cset^{\A_{d-1}}$ into $\Cset^{\Ad}$ such that $\iota_p(e_q) = e_{\iota_p(q)}$, and by 
$H_p$ the hyperplane of $\Cset^{\Ad}$ which is the image of this embedding.

\smallskip

From the defining formulas, for all $p \in  \Ad, q \in \A_{d-1}$ , the hyperplane $H_p$ is invariant under $L_{\iota_p(q)}$ and we have

$$ \iota_p \circ L_q = L_{\iota_p(q)} \circ \iota_p. $$

In order to avoid confusion, we denote by $Q'_{\alpha}$ the hermitian form on $\Cset^{\A_{d-1}}$ denoted by $Q_{\alpha}$ previously, and keep the notation $Q_{\alpha}$ for 
the hermitian form on  $\Cset^{\Ad}$.

\begin{lemma} \label{restrict}
For any $p \in \Ad$ , the restriction  to $H_p$ of the form $Q_{\alpha}$ on   $\Cset^{\Ad}$ is equal to the image  of  $Q'_{\alpha}$ under $\iota_p$. 
\end{lemma}

\begin{proof}
This is clear from subsection \ref{operators}
\end{proof}

When $\rho^{d+1} \ne 1$, we denote by $H'_p$ the $1$-dimensional subspace which is  the \linebreak $Q_{\alpha}$-orthogonal of $H_p$. As $Q_{\alpha}$ is non  degenerate and 
$\bigcap_{p \in \Ad} H_p = \{0\}$, $\Cset^{\Ad}$ is the direct sum of the $H'_p$. 

\smallskip

When moreover $\rho^d \ne 1$, the restriction of $Q_{\alpha}$ to each $H_p$ is non degenerate by the lemma. Therefore  $\Cset^{\Ad}$ is the direct sum of $H_p$ and $H'_p$. From the formulas for $L_q$, $q\ne p$, we see that the line $H'_p$ is point wise fixed under all $L_q$, $q\ne p$. 

\subsubsection{A result on stabilizers}

\begin{proposition}\label{stabilizers}
Let $d \ge 3$, $Q$ a non degenerate hermitian form on $\Cset^d$, $u_1, u_2  \in \Cset^d $  linearly independent vectors such that $Q(u_1) =Q(u_2) \ne 0$.
 Let $SU(Q)$ be the special unitary group of $Q$, and  , for $j=1,2$, let $G_j$ be the stabilizer of $u_j$ in  $SU(Q)$. Then the smallest closed subgroup containing 
 $G_1 \cup G_2 $ is $SU(Q)$.
\end{proposition}

\begin{proof}
Let $G$ be the smallest closed subgroup of $SU(Q)$ containing  $G_1 \cup G_2 $. Let $c$ be the common value of the $Q(u_j)$. It is sufficient to prove that $G$ acts transitively on
 $\{Q(u)=c\}$. Indeed, assume this is true, and let $h$ be an element of $SU(Q)$; there exists $g\in G$ such that $g(u_1) = h(u_1)$; then $g^{-1}h \in G_1$ and $h \in G$.
 To prove that $G$ acts transitively on $\{Q(u)=c\}$, we observe that, as $\{Q(u)=c\}$ is connected\footnote{Since we are assuming that $c=Q(u_1)=Q(u_2)\ne 0$, the desired connectedness follows from Witt's theorem (see also the claim in the proof of Lemma \ref{l.large-interior} below).}, it is sufficient to show that the orbits of $G$ have non empty interior  
 (and so are open)  in  $\{Q(u)=c\}$. 
 
 \smallskip
 
 Let $u_0$ be a vector such that $Q(u_0)=c$. 
 \begin{lemma}\label{l.large-interior}
 If either $Q(Gu_0,u_1)$ or $Q(Gu_0, u_2)$ have non empty interior in $\Cset$, then $Gu_0$ has non empty interior  
  in  $\{Q(u)=c\}$. 
  \end{lemma}
  \begin{proof}
  Recall the following 
  
 \hspace{2cm}   {\bf Fact}:  for $p+q \ge 2$ and $b \ne 0$, $SU(p,q)$ acts transitively on \linebreak  $\{ \sum_1^p |z_i|^2 - \sum_{p+1}^{q+1} |z_i|^2 = b\} \subset \Cset^{p+q}$.
 \smallskip
 
 Assume for instance that   $Q(Gu_0, u_1)$ has non empty interior in $\Cset$. Let $W$ be a non-empty open set which is contained in $Q(Gu_0, u_1)$ and is disjoint from the  circle
 $\{|z|=Q(u_1)\}$. For any $w \in W$, the intersection of $\{Q(u)=Q(u_1)\}$ with $\{Q(u,u_1) = w \}$ consists of vectors of the form $u = \alpha u_1 +v$, 
 with $\alpha = \frac w{Q(u_1)}$ , $Q(v,u_1)=0$, $ Q(v) = (1- |\alpha|^2) Q(u_1)$.  As $|\alpha| \ne 1$, it follows from the fact recalled above that the intersection of
 $\{Q(u)=Q(u_1)\}$ with $\{Q(u,u_1) = w \}$ is contained in $Gu_0$.
  \end{proof}
  
\medskip
  
\begin{lemma}
Let $i \in \{1,2\}$. If  $Q(Gu_0,u_i)$  is not contained in the circle $\{|z| =c\}$, then $Q(Gu_0,u_{3-i})$ has non empty interior in $\Cset$.
\end{lemma} 
\begin{proof} 
We may assume for instance that $u_0 = \alpha u_1 +v_0$ with $|\alpha| \ne 1$,  $Q(v_0,u_1)=0$, $ Q(v_0) = (1- |\alpha|^2) Q(u_1) \ne 0$. Write $u_2 = \beta u_1 + v_2$ with 
$Q(v_2,u_1)=0$, $v_2 \ne 0$. For $g \in G_1$, we have 
$$ Q(g.u_0, u_2) = \alpha \bar{\beta} Q(u_1)+ Q(g.v_0, v_2) .$$
From the fact recalled above, the set $G_1 v_0$ consists of the vectors $v$ orthogonal to $u_1$ satisfying $Q(v)= (1- |\alpha|^2) Q(u_1) $. Any linear projection of this set on $\Cset$
has nonempty interior, which proves the assertion of the lemma. 
\end{proof}

\smallskip

We can now end the proof of  the proposition. In view of the two lemmas above, we know that $Gu_0$ has non empty interior in $\{Q(u)=c\}$ except perhaps if both $Q(Gu_0,u_1)$ 
and $Q(Gu_0, u_2)$ are contained in the circle $\{|z| =c\}$. We will now prove that this exceptional case is impossible. Write as before  $u_0 = \alpha u_1 +v_0$ ,
 $u_2 = \beta u_1 + v_2$, with  $Q(v_0,u_1)=  Q(v_2,u_1)=0$,  $v_2 \ne 0$. Exchanging $u_1, u_2$ if necessary, we may assume that $v_0 \ne 0$. We may also assume that $|\alpha | =1$, $Q(v_0) =0$. Choose a $2$-dimensional subspace $E$ of the hyperplane $H$ orthogonal to $u_1$ with the following properties:
 \begin{itemize}
 \item the subspace $E$ contains $v_0$.
 \item The restriction of $Q$ to $E$ is non degenerate.
 \item The orthogonal projection of $v_2$ on $E$ is $\ne 0$.
 \end{itemize}

Such a choice is possible because  the last two conditions are open and dense amongst $2$-dimensional subspaces of $H$ containing $v_0$. Choose a basis $e,f$ of $E$ such that
$v_0 = e+f$ and $Q(xe +yf) = |x|^2 -|y|^2$. For any $a,b \in \Cset$ with $  |a|^2 -|b|^2 =1$, we can find $g \in G_1$ such that
$$ g.e = ae + bf, \quad g.f =  \bar b e + \bar a f .$$
Therefore any vector of the form $z e + \bar z f $, with $z \in \Cset^*$ belongs to $G_1 v_0$.

\smallskip

Let $s e + t f \ne 0$ be the  orthogonal projection of $v_2$ on $E$. Then $z \bar s - \bar z \bar t $ belongs to $Q(G_1 v_0 , v_2)$ for any  $z \in \Cset^*$. As the set  $\{ z \bar s - \bar z \bar t , z \in \Cset^*\}$ contains a straight segment in $\Cset$, the set $Q(G_1u_0, u_2) = \alpha \bar {\beta} c + Q(G_1 v_0 , v_2)$ is not contained in the circle $\{|z| =c\}$. 

\smallskip

This concludes the proof of the proposition.
\end{proof}

\subsubsection{Application}

\begin{proposition}\label{propind1}
Assume that $d \ge 3$, $\rho ^d, \rho^{d+1} \ne 1$. Assume also that   the intersection of the group generated by the operators 
$L_q$,  $q \in \A_{d-1}$, on $\Cset^{\A_{d-1}}$ with the special unitary group $SU(Q'_{\alpha})$ is dense (resp. Zariski dense)  in  $SU(Q'_{\alpha})$ . 
Then  the intersection of the group generated by the operators  $L_p$, $p \in \Ad$,  on $\Cset^{\A_{d}}$ with the special unitary group $SU(Q_{\alpha})$ is dense 
(resp. Zariski dense)  in  $SU(Q_{\alpha})$ .
\end{proposition}

\begin{proof}
Let $p_1, p_2$ be two distinct elements of $\Ad$. Denote by $u_1, u_2$ generators of $H'_{p_1}, H'_{p_2}$ respectively, satisfying $Q_{\alpha}(u_1) = Q_{\alpha}(u_2)\ne 0$. Such a choice is possible because the restrictions of $Q_{\alpha}$ to  $H_{p_1}, H_{p_2}$ have the same signature by lemma \ref{restrict}. By the assumption of the proposition, for $i \in \{1,2\}$, the intersection of the group generated by the operators  $L_p$, $p \in \Ad$, $p\ne p_i$,  with the special unitary group $SU(Q_{\alpha})$ is dense (resp. Zariski dense)
in the stabilizer $G_i$ of $u_i$ in $SU(Q_{\alpha})$. By Proposition \ref{stabilizers}, the smallest closed group containing $G_1 \cup G_2$ is $SU(Q_{\alpha})$. We thus obtain the conclusion of the proposition .
\end{proof}

%%%%%%%%%%%%%%%%%%%%%%%%%%%%%%%%%%%%%%%
\subsection{The induction step in the non exceptional degenerate case}\label{ssdegnonexc}
%%%%%%%%%%%%%%%%%%%%%%%%%%%%%%%%%%%%%%%

\subsubsection{The setting}

We assume in this subsection that $d \geq 4$ and that $(d+1) \alpha$ is an integer. Therefore the hermitian form $Q_{\alpha}$ 
on $\Cset^{\Ad}$ is degenerate. As $0<\alpha < \frac 12$, $d \alpha$ is not an integer and the hermitian form $Q'_{\alpha}$ on $\Cset^{\A_{d-1}}$ is non-degenerate. As the restriction of $Q_{\alpha}$ to each hyperplane $H_p$ is isomorphic to $Q'_{\alpha}$ (Lemma \ref{restrict}), the kernel of $Q_{\alpha}$ has dimension $1$.

\begin{remark}
In the case $d=3$, we must have $\alpha = \frac 14$; then the induction hypothesis (see below) is not satisfied.
\end{remark} 

\par
As $Q_{\alpha}$ is invariant under each $L_p$ the kernel of $\Qa$ is invariant under the $L_p$. 
But the eigenvalues of $L_p$ are $1$ with multiplicity $(d-1)$ and $-\zeta$ with multiplicity $1$, 
and the eigenvector associated to the eigenvalue $-\zeta$ is $e_p$, which is not an eigenvector of $L_q$
 for $q \ne p$. We conclude that the kernel of $\Qa$ is pointwise fixed by each $L_p$.

\par
We make the following {\bf  induction hypothesis}: on $\Cset^{\A_{d-1}}$, the intersection of the subgroup generated by the $L_p$, $p \in \A_{d-1}$ with the special unitary group $SU(\QA)$ is  dense for the ordinary topology
(resp. Zariski dense) in $SU(\QA)$.
 
\subsubsection{The induction step}

As the restriction of $\Qa$ to each $H_p$ is non-degenerate, the kernel $\Cset e$ of $\Qa$ is not contained in any $H_p$. 

\par
Let us denote by $SU^*(\Qa)$ the subgroup of $GL(\Cset^{\Ad})$ formed by linear automorphisms which preserve $\Qa$, fix $e$ (and not simply the line $\Cset e$) and have determinant $1$. If one writes these
automorphisms in the basis $(e, e_{3-d}, \ldots, e_{d-1})$ (using that $\Cset^{\Ad} = \Cset e \oplus H_{1-d}$),
the matrix takes a block triangular form
$$ M =  \left ( \begin{array}{cc}1 & v \\ 0 & g \end{array} \right  ),$$
with $g \in SU(\QA)$. Let $D$ be the subgroup of $SU^*(\Qa)$ formed of automorphisms whose matrix in the selected basis satisfies $v=0$.

\begin{proposition}\label{nointermediatesubgroup}
A  subgroup of $SU^*(\Qa)$ which contains $D$ is equal to $D$ or to $SU^*(\Qa)$.
\end{proposition}

\begin{proof}
Let $D'$ be such a subgroup . We identify $H_{1-d}$ to $\Cset^{\A_{d-1}}$ through $\iota_{1-d}$. Let 
$$ V(D'):= \{ v \in (\Cset^{\A_{d-1}})^*,  \left ( \begin{array}{cc}1 & v \\ 0 & 1 \end{array} \right  ) 
\in D' \}.$$
Clearly $V(D')$ is an additive subgroup of $(\Cset^{\A_{d-1}})^*$. Moreover, as 

$$  \left ( \begin{array}{cc}1 & 0 \\ 0 & g^{-1} \end{array} \right  )  \left ( \begin{array}{cc}1 & v \\ 0 & 1 \end{array} \right  )  \left ( \begin{array}{cc}1 & 0 \\ 0 & g \end{array} \right  ) = \left ( \begin{array}{cc}1 & v.g \\ 0 & 1 \end{array} \right  ),$$

$V(D')$ is invariant under the natural action of $SU(\QA)$ on  $(\Cset^{\A_{d-1}})^*$. In view of the lemma below, $V(D')$ must be equal to $\{0\}$ or $(\Cset^{\A_{d-1}})^*$, which corresponds to $D'=D$ and $D' = SU^*(\Qa)*$.
\end{proof}

\begin{lemma}
Let $Q$ be a non-degenerate hermitian form on $\Cset^N$, $N \geq 3$. The only additive subgroups of $\Cset^N$ which are invariant under $SU(Q)$ are $\{0\}$ and $\Cset^N$.
\end{lemma}
\begin{proof}
We first observe that , as $N \geq 3$, the orbit $SU(Q).v_0$ of a vector $v_0$ is equal to
\begin{itemize}
\item $\{0\}$ if $v_0 =0$;
\item $\{ v\ne 0 , Q(v) =0\}$ if $v_0 \ne 0$, $Q(v_0) =0$;
\item $\{Q(v)=c\}$ if $Q(v_0) = c \ne 0$.
\end{itemize}
Indeed, Witt's theorem implies that the orbit $U(Q).v_0$ is as stated. As $N\geq 3$, there exists
a vector $v_1$ in this orbit with at least one coordinate vanishing 
(in an orthogonal basis for $Q$) . But then the image of the stabilizer of $v_1$ (in $U(Q)$) by the determinant map is the full unit circle. This means that the orbits  $U(Q).v_1$ and $SU(Q).v_1$ are equal.
\par
Let $V$ be an additive subgroup of $\Cset^N$ which is also $SU(Q)$-invariant. It is sufficient to show that , if $V$ contains a non-zero vector, then $V$ has non-empty interior. Let $v_0 \in V$,
$v_0 \ne 0$. Then $V$ contains $V_0:=\{ v\ne 0, Q(v) = Q(v_0)  \}$. The set $V_0 -v_0$ is also contained in $V$. The map 
$$ v \to Q(v-v_0) = 2 Q(v_0) -2 \Re Q(v,v_0)$$
is not constant in a neighborhood of $v_0$ in $V_0$, hence its image contains a non-trivial interval. This implies that $V$ has non-empty interior.
\end{proof}

\begin{corollary}\label{cordegnonexcept}
Under the induction hypothesis stated above, the intersection of the subgroup generated by the $L_p$ with $SU(\Qa)$ is dense (resp. Zariski dense) in $SU^*(\Qa)$.
\end{corollary}

\begin{proof}
Let $ G$ be the closure (resp. the Zariski closure) of the intersection of the subgroup generated by the $L_p$, $1-d \leq p \leq d-1$ with $SL(n,\Cset)$. We have seen earlier that $G$ is contained in $SU^*(\Qa)$.
\par

Let $G'$ be the closure (resp. the Zariski closure) of the intersection of the subgroup generated by the $L_p$, $3-d \leq p \leq d-1$ with $SL(n,\Cset)$. We have seen earlier that $G' \subset D$.
It follows from the induction hypothesis that $G' =D$. As $L_{1-d} \circ L_{d-1}^{-1}$ has determinant $1$ but does not preserve $H_{1-d}$, $G$ is not equal to $D$. It follows from the proposition that $G = SU^*(\Qa)$.  
\end{proof}

\subsection{From the degenerate case to the non-degenerate case}\label{ssdegnondeg}

\subsubsection{The setting}\label{ssSettingdegnondeg}
We assume in this subsection that $d\geq 5$ and that $d\alpha$ is an integer. Therefore the hermitian form $\Qa$ is non degenerate , but the restrictions of $\Qa$ to each hyperplane 
$H_p$, which are isomorphic to $\QA$,  are degenerate. It means that the $\Qa$-orthogonal of $H_p$ is a line $\Cset w_p$ contained in $H_p$. 
\par
We make the following {\bf  induction hypothesis}:  The intersection of the subgroup generated by the $L_p$, $p \in \A_{d-1}$ with $SL( \Cset^{\A_{d-1}})$ is  dense for the ordinary topology
(resp. Zariski dense) in the subgroup $SU^*(\QA)$ defined in the previous subsection.

\subsubsection{ Stabilizers }

Let $Q$ be a non-degenerate non-definite hermitian form on $\Cset^N$. 
Let $w$ be a nonzero vector such that $Q(w)=0$.
 Let $H$ be the hyperplane $Q$-orthogonal to $\Cset w$.
It contains $\Cset w$. Denote by $Q'$ the restriction of $Q$ to $H$, 
which is degenerate. Choose a vector $w'$ such that $Q(w,w') =1$. 
Denote by $H'$ the orthogonal of the plane $\Cset w \oplus \Cset w'$, and
by $Q''$ the restriction of $Q$ to $H'$, which is non-degenerate.  We have  direct sums

$$ \Cset^N = H \oplus\Cset w', \qquad H = \Cset w \oplus H' . $$

%Choose a subspace $H'$ such that $H = \Cset w \oplus H'$, and a vector 
%$e$  such that $H \oplus \Cset e = \Cset^N$ and $<w,e>$ is real.

\par

Write ${\rm Stab}(w)$ for the stabilizer of $w$ in $SU(Q)$. The matrix of an element of 
${\rm Stab}(w)$ in the decomposition $\Cset w \oplus H' \oplus \Cset w'$ of $\Cset^N$ takes a block-triangular form
\begin{equation}\label{eqform}
  M=  \left ( \begin{array}{ccc}1 & v & t \\ 0 & g & h \\ 0 & 0 & \omega \end{array} \right  )
  \end{equation}

However, not all such matrices $M$ are associated to elements of $SU(Q)$. Define a one-dimensional subgroup
$$  K := \lbrace  \left ( \begin{array}{ccc}1 & 0 & t \\ 0 & 1 & 0 \\ 0 & 0 & 1 \end{array} \right  ),
\Re t =0\rbrace. $$

In the next proposition, the group $SU^*(Q')$ was defined in the last subsection.
 
\begin{proposition}\label{propnotsplit}
We have an exact sequence
$$ 1 \longrightarrow K  \longrightarrow {\rm Stab}(w)  \longrightarrow SU^*(Q')  
\longrightarrow 1$$
The homomorphism $\theta$ from   ${\rm Stab}(w)$ to $SU^*(Q')$ is induced by restriction to $H$. The exact sequence is {\bf not} split.
\end{proposition}

\begin{proof}
If a matrix of the form (\ref{eqform}) is unitary, we must have $\omega =1$ because the  scalar product $Q(w,w')$ is preserved.

As $\omega \equiv 1$, the homomorphism $\theta$ takes values in 
$SU^*(Q')$. 
It is onto by Witt's theorem.

An elementary computation shows that the kernel of $\theta$ is equal to $K$ (see also below).

It remains to show that $\theta$ has no section.  Assume by contradiction that such a section $\sigma$ exists. Consider

%Consider first 

$$ u(v):= \sigma \left( \begin{array}{cc} 1&v\\0&1 \end{array}\right) =  
\left( \begin{array}{ccc} 1&v&t(v)\\0&1 & h(v) \\ 0&0&1  \end{array} \right) . $$

As $u(v)$ is unitary, we have, for all $v \in (H')^*, x \in H'$

$$ Q(u(v)x,u(v)w') = Q(v(x)w +x,\, w' + h(v) + t(v) w) = v(x) + Q(x,h(v)) =0.$$

%As $\sigma$ is a homomorphism, we have $t(gg') = t(g) + t(g')$. This implies that $t \equiv 0$ because the special unitary group $SU(Q'')$ is simple. 
%\par
%Writing that u(g) is unitary, we have $<x,w'> =0$ for all $x \in H'$, hence

%$$< g.x, h(g) + w'> = <g.x,h(g)> =0.$$
%As $Q''$ is non-degenerate, we must have $h(g) \equiv 0$.

This determines $h$ as a semi-linear isomorphism from $(H')^*$ to $H'$.
On the other hand, as $\sigma$ is a homomorphism, we have
$$ t(v+v') = t(v) +t(v')  + v(h(v')),$$
for all $v,v' \in (H')^*$. Therefore $v(h(v'))$ is a symmetric function of $v, v'$.
As $h$ is $\mathbb{C}$-antilinear, we should have $v(h(v')) \equiv 0$. This is not true since $h$ is an isomorphism.

\end{proof}

\par
We now obtain in our particular setting:

\begin{corollary}\label{Zar1}
Let $p \in \A_d$. The intersection with $SU(\Qa)$ of the subgroup 
generated by the $L_q$, 
$q \ne p$, is Zariski dense  in the stabilizer ${\rm Stab}(w_p)$.
\end{corollary}

\begin{remark}
Probably one doesn't get the density in the usual topology, even if we started from this form of the induction hypothesis.
\end{remark}

\begin{proof}
We already know that $L_q(w_p) = w_p $ for $q \ne p$. Therefore the Zariski closure 
 $G_p$ of the intersection with $SU(\Qa)$ of the subgroup generated 
by the $L_q$, $q \ne p$ is a Zariski closed subgroup contained in 
${\rm Stab}(w_p)$.  On the other hand, the induction hypothesis 
(applied to the restriction of $\Qa$ to $H_p$, which is isomorphic to $\QA$) 
implies that restriction to $H_p$ induces an homomorphism of $G_p$ onto
$SU^*(\QA)$.

\par

The kernel of the homomorphism from $G_p$ onto $SU^*(\QA)$ is a Zariski closed subgroup of $K$ hence it is either equal to $K$ (in which case $G_p = {\rm Stab}(w_p)$) or to $\{1\}$. (If this subgroup is only closed for the usual topology, it could be an infinite  discrete subgroup of $K$). But the second case is impossible since the exact sequence of the proposition is not split.

\end{proof}

\subsubsection{More stabilizers}

Let $Q$ be a non-degenerate non-definite hermitian form on $\Cset^N$. 
Let $w$ be a nonzero vector such that $Q(w)=0$. Let ${\rm Stab}(w)$ be the stabilizer of $w$ in $SU(Q)$. From Witt's theorem, one gets

\begin{lemma}\label{trans1}
($N \geq 3$) For any $c \in \Cset$, $c\ne 0$, the group ${\rm Stab}(w)$ acts transitively on
$$N(w,c):= \{ u \in \Cset^N , Q(u) =0, Q(u,w) =c\}.$$
\end{lemma}
\begin{proof} 
Let $u \in N(w,c)$. It is sufficient to see that the determinant of an element $g \in SU(Q)$ which stabilizes $w$ and $u$ can have any determinant of modulus one. This is clear since  the restriction of $Q$ to the orthogonal of $<u,w>$ is non-degenerate and the restriction of $g$ to this subspace  is any unitary matrix.
\end{proof}
\begin{lemma}\label{trans2}
Let $F$ be a nontrivial linear subspace of $\Cset^N$. Assume that the restriction of $Q$ to $F$ is non-degenerate.
\begin{enumerate} 
\item If the restriction of $Q$ to $F$ is indefinite, any translate $u+F$ intersects $\{Q=0\}$.
\item Assume that the restriction of $Q$ to $F$ is positive definite (resp. negative definite). Then a translate $u+F$, $u \in F^{\bot}$, intersects $\{Q=0\}$ iff $Q(u)\leq 0$
(resp. $Q(u)\geq 0$).
\end{enumerate}
\end{lemma}
\begin{proof}
\begin{enumerate}
\item The real-valued function $f \mapsto Q (u+f)= Q(u) + Q(f)  + 2\Re Q(u,f) $ takes on $F$ arbitrarily large positive and negative values, hence must vanish somewhere.
\item One has now $Q(u+f) = Q(u) + Q(f) \geq Q(u)$; the conclusion follows.
\end{enumerate}
\end{proof}
\begin{lemma}\label{trans3}
Assume $N=2$. Let $w, w'$ be a basis of $ \Cset^2$, such that $Q(w) = 0$. 
Then $z \to Q(w + zw')$ takes positive and negative values in any neighbothood of $0$.
\end{lemma}

\begin{proof}
Indeed, one has $Q(w+zw') = |z|^2 Q(w') + 2\Re (z Q(w',w))$ with $Q(w',w)\ne 0$ as $Q$ is non degenerate and $Q(w) =0$. 
\end{proof}

\begin{proposition}\label{stabdeg1}
Let $(w_1,\ldots,w_N)$ be a basis of $\Cset^N$ with $Q(w_i) =0$ for $1\leq i \leq N$, $<w_i,w_j> \ne 0$ for $1 \leq i  \ne j \leq N$. Let $G_i$ be the stabilizer of $w_i$ in $SU(Q)$. Then the smallest closed subgroup containing $G_1,\ldots,G_N$ is $SU(Q)$.
\end{proposition}

\begin{proof}
Let $G$ be the smallest subgroup containing $G_1,\ldots,G_N$. It is sufficient to show that $G$ acts transitively on $\mathfrak Q:= \{Q(u) = 0, u\ne 0\}$. As this last set is connected, it is sufficient to show that any orbit of $G$ in $\mathfrak Q$ has non empty interior.

\par
% Denote by $\iota$ the linear 
%automorphism of $\Cset^N$ defined by $\iota (u) := (<u,w_i>)_{1\leq i \leq N}$ and by $\bar Q$ the form $Q \circ \iota^{-1}$.

%\begin{lemma}
%Let $u_0 \in \mathfrak Q$. There exist  indices $1 \leq i \ne j \leq N$ such that the image of $G.u_0$ by the map $u \mapsto (<u,w_i>,<u,w_j>)$ has non empty interior in $\Cset^2$.
%\end{lemma}

\begin{lemma}
Let $u_0 \in \mathfrak Q$. There exists an  index $1 \leq i \leq N$ such that the image of $G.u_0$ by the map $u \mapsto Q(u,w_i)$ has non empty interior in $\Cset$.
\end{lemma}

\begin{proof}
We first claim that there exist distinct indices $i$, $j$ and $u_1 \in G.u_0$ such that $Q(u_1,w_i)\ne 0$, $Q(u_1,w_j) \ne 0$.
\par

Let $i$ be an index such that $c_i:= Q(u_0,w_i) \ne 0$.  Let $j \ne i$. As $Q(w_i,w_j) \ne 0$, there exists $\lambda \in \Cset$ such that $\Re (\bar \lambda c_i) =0$ and
$Q(u_0 + \lambda w_i,w_j) \ne 0$. Take $u_1 := u_0 + \lambda w_i$. One has 
$Q(u_1) = 0$, $Q(u_1,w_j)=: c_j \ne 0$, $Q(u_1,w_i) =  c_i$ and $u_1 \in G_i.u_0$ by Lemma \ref{trans1}. This proves the claim.
\par
Let $F$ be the codimension $2$ subspace of $\Cset^N$ orthogonal to $w_i,w_j$.
As $Q(w_i,w_j) \ne 0$, the restriction of $Q$ to $F$ is nondegenerate.
\begin{itemize}
\item If the restriction of $Q$ to $F$ is indefinite, Lemma \ref{trans2} implies that for any $c'_i$ close to $c_i$
there exist $u_2  \in \mathfrak Q $  such that $Q(u_2,w_i) =  c'_i, Q(u_2,w_j)=c_j $. By Lemma \ref {trans1}, one has $u_2 \in G_j.u_1$. This proves the assertion of the lemma in this case.
\item 
Assume that the restriction of $Q$ to $F$ is positive definite (the negative case is symmetric). Then the restriction of $Q$ to $F^{\bot}$ is indefinite. Identify $F^{\bot}$ to $\Cset^2$ through $u \mapsto (Q(u,w_i),Q(u,w_j))$. We have $Q(c_i,c_j) \leq 0$ 
by Lemma \ref{trans2}. If $Q(c_i,c_j) < 0$, we proceed as in the first case to get the conclusion of the lemma (using again Lemma \ref{trans2}). If $Q(c_i,c_j)=0$, by Lemma \ref{trans3} there exists $c'_j$ close to $c_j$ such that $Q(c_i , c'_j) < 0$.
Then there exists $u'_1 \in \mathfrak Q$ such that $Q(u'_1,w_i)=c_i, Q(u'_1,w_j) = c'_j$.
One has $u'_1 \in G_i.u_1$ from Lemma \ref{trans1}. Then the end of the argument is the same than for $Q(c_i,c_j) <0$.
\end{itemize}
\end{proof}

\medskip

From the lemma, we may assume $c_i := Q(u_0,w_i) $ is different from $0$ and that a small neighborhood of $c_i$ in $\Cset$ is contained in the   image of $G.u_0$ by the map $u \mapsto Q(u,w_i)$. Let $u \in \mathfrak Q$ be close to $u_0$.  There exists $u_1 \in G.u_0$ such that $Q(u_1,w_i) = Q(u,w_i) (\ne 0)$. By Lemma \ref{trans1}, one has $u \in G_i.u_1 \subset G.u_0$.
\end{proof}

Putting together Proposition \ref{stabdeg1} and Corollary \ref{Zar1}, we obtain

\begin{corollary}\label{cordegnondeg}
The intersection with $SU(Q_{\alpha})$ of the subgroup generated by the $L_p$ is 
Zariski dense in $SU(Q_{\alpha})$.
\end{corollary}

\begin{proof}
We apply Proposition \ref{stabdeg1}, taking for $w_p$ ($p \in \Ad$) a generator of the orthogonal $H'_p$ of $H_p$.  We have $Q_{\alpha} (w_p) =0$. If we had $Q_{\alpha}(w_p,w_q) = 0$ for some distint $p,q \in \Ad$, the vector $w_p$ would belong to $H_p \cap H_q$, and two of its coordinates in the canonical basis would vanish. But we know from the diagonalisation formulas of subsection \ref{ssdiag} that it is not so.
From Corollary \ref{Zar1}, the Zariski closure of the intersection with $SU(Q_{\alpha})$ of the subgroup generated by the $L_p$ contains the stabilizer of each $w_p$.
Therefore it is equal to $SU(Q_{\alpha})$.
\end{proof}

%%%%%%%%%%%%%%%%%%%%%%%%%%%%%
\subsection{Exceptional case I: $\mathbf {\alpha = 1/3}$} \label{ss13}
%%%%%%%%%%%%%%%%%%%%%%%%%%%%

At this stage, we can conclude that the Zariski closure of the intersection with 
$SU(Q_{\alpha})$ of the subgroup generated by the $L_p$ is equal to $SU(Q_{\alpha})$ when $\alpha \ne \frac 16, \frac 14, \frac 13$ and $(d+1) \alpha$ is not an integer (so that $Q_{\alpha}$ is non-degenerate). We may even replace Zariski closure by closure for the usual topology when the form is definite, i.e  $(d+1) \alpha <1$. Indeed it is sufficient to proceed by induction on the dimension $d$, starting with the results of Subsection \ref{ssd2}, and applying successively either Proposition \ref{propind1}, Corollary \ref{cordegnonexcept}, or Corollary \ref{cordegnondeg}.

\par
In the two exceptional cases $\alpha = \frac 16, \frac 14$ , the group generated 
by $L_1$, $L_{-1}$ in dimension $2$ is a finite group so it is not a basis for a successful induction. These two cases will be dealt with in later subsections. 
However, the case $\alpha = {\frac 13}$ does not need any supplemental work.

\par
Remember that $Q_{\frac 13}$ is degenerate for $d=2$. The kernel of $Q_{\alpha}$ is generated by $e:= e_{-1} + j e_1$ and fixed by $L_{-1}$ and $L_1$. It was observed  in Subsection \ref{ssd2} that the intersection of the group generated by  $L_{-1}$ and $L_1$ with $SL(2,\Cset)$ contains a parabolic matrix. The determinant of $L_{-1}$ and $L_1$ is $-\zeta$, a sixth root of unity.
This is sufficient to show that the Zariski closure $G$ of the intersection of the group generated by  $L_{-1}$ and $L_1$ with $SL(2,\Cset)$ is the stabilizer of $e$ in 
$SL(2,\Cset)$. Indeed, $G$ is contained in this stabilizer. As $\Rset$-Lie groups, the stabilizer of $e$ in $SL(2,\Cset)$ has only three types of Zariski closed  subgroups: the two trivial subgroups\footnote{I.e., $\{\textrm{Id}\}$ and the full stabilizer itself.}, and (given any vector $f$ independent of $e$) the subgroup $G_f$ of the stabilizer formed of elements $g$ such that $g.f -f$ is a {\bf real} multiple of $e$ (there is a one-parameter family, parametrized by the $1$-dimensional real projective space, of such subgroups). Here, the existence of a parabolic element guarantees that $G$ is not reduced to the identity. It cannot be of the intermediate form, because conjugating by powers of $L_1$ an element $g$ such that $g.f -f = e$, we get elements $g'$ such that $g'.f -f = \omega e$ for any sixth root of unity. Therefore $G$ is equal to the full stabilizer.

\par
Thus the argument from Subsubsection \ref{ssdegnondeg} can be used to conclude the desired result (even though the dimension $d$ is not as high as assumed there).

%%%%%%%%%%%%%%%%%%%%%%%%%%%%%
\subsection{Exceptional case II: $\mathbf{\alpha = 1/4}$} \label{ss14}
%%%%%%%%%%%%%%%%%%%%%%%%%%%%

We assume in this subsection that $\alpha = \frac 14$. Then, we have seen in Subsection \ref {ssd2} that 
the group generated by $L_{-1}$ and $L_1$ is finite. For $d=3$, the form 
$Q_{\frac 14}$ is degenerate, but the Zariski closure $G$ of the intersection with $SL(3,\Cset)$ of the group generated by $L_{-2}$, $L_0$ and $L_2$ is  
strictly smaller than the group $SU^*(Q_{\frac 14})$ described in Subsection 
\ref{ssdegnonexc} (see below). It is only from dimension $4$ that we get a "big" group generated by the $L_p$.

\par
Consider first the case $d=2$. It can be checked that the group $\Gamma$ generated by $L_{-1}$ and $L_1$ has order $96$. The property of $\Gamma$ that will be useful in the sequel is 

\begin{lemma}\label{14lem1}
The representation of $\Gamma$ on $\Cset^2 \simeq \Rset^4$ induced by the inclusion $\Gamma \subset U(Q_{\frac 14})$, is irreducible over $\Rset$.
\end{lemma}

\begin{proof} It is clear that $L_1$ and $L_{-1}$ do not have a common eigenvector, therefore the representation is irreducible over $\Cset$. But $\Gamma$ contains the scalar multiplicatons by the fourth roots of unity, hence any $\Rset$-subspace invariant under $\Gamma$ has to be complex.
\end{proof}

We can now determine what is the Zariski closure $\Gamma_3$ of the group generated by the $L_p$ for $d=3$. As in Subsection \ref{ssdegnonexc}, let us denote by $e$ a generator of the kernel of $Q_{\frac 14}$, and write the operators in the basis $(e,e_0,e_2)$ (using that all coordinates of $e$ are non zero). All $L_p$ fix $e$ so the matrices take the block-form 

$$\left( \begin{array}{cc} 1 & v \\ 0 & \gamma \end{array} \right)$$

Here, the $2 \times 2$ block $\gamma$ will vary exactly in the finite group $\Gamma$ mentioned above (hence $G$ will not contain $SU^*(Q_{\frac 14})$).
Therefore we have an exact sequence
$$ 1 \longrightarrow K \longrightarrow \Gamma_3 \longrightarrow \Gamma \longrightarrow 1,$$
where the kernel $K$ is the Zariski closed subgroup of $(H_{-2})^*$ formed of 
those $v$ such that $\left( \begin{array}{cc} 1 & v \\ 0 & {\mathbf 1} \end{array} \right)$ belongs to $\Gamma_3$.
Note that  $K$ is not only a Zariski closed subgroup of $(H_{-2})^*$, it is also
invariant under the representation of $\Gamma$ dual to that induced by the 
inclusion of $\Gamma$ in $U(Q'_{\frac 14})$: indeed, conjugating an element of
 the kernel by an element of $\Gamma_3$ with diagonal block $\gamma$
  changes $v$ into $v.\gamma^{-1}$. By Lemma \ref {14lem1} this representation
   is irreducible over $\Rset$. We conclude that $K$ must be equal to either 
   $\{0\}$ or $(H_{-2})^*$.

\begin{lemma}\label{14lem2}
The kernel $K$ is equal to $(H_{-2})^*$.
\end{lemma}

\begin{proof}
Otherwise, $K$ would be trivial and $\Gamma_3$ would be finite. As any representation of a finite group is semi-simple, there would exist a $\Gamma_3$-invariant complex $2$-plane supplemented by $\Cset e$. But then the transposed matrices $^t L_{-2}$, $^t L_{0}$, $^t L_{2}$ would have a common eigenvector. This is clearly not the case.
\end{proof}

Having described $\Gamma_3$, we now go to the case $d=4$. To represent the operators $L_p$, we chose a basis $w_{-3}, f_{-1}, f_1, w_3$ with the following properties:
\begin{itemize}
\item $w_{-3}$ is a generator of the orthogonal of the hyperplane $H_{-3}$;
\item $w_{3}$ is a generator of the orthogonal of the hyperplane $H_{3}$;
\item $Q_{\frac 14}(w_{-3} , w_3) =1$
\item the subspace $F$ generated by $f_{-1}, \, f_1$ is the orthogonal of the subspace $W$ generated by $w_{-3} , w_3$; note that the restriction of $Q_{\frac 14}$ to $W$ has signature $(1,1)$, hence $W$ and $F$ are indeed transverse and the restriction of $Q_{\frac 14}$ to $F$ has signature $(2,0)$.
\item $f_{-1},f_1$ form an orthonormal basis of $F$.
\end{itemize}

\par
Observe that the first three vectors form a basis of $H_{-3}$, and the last three a basis of $H_3$.

Consider the  Zariski closure of the subgroup generated  by $L_{-1}, \, L_1, \,L_3$. 
An element in this group preserves the hyperplane $H_{-3}$ and its restriction to $H_{-3}$ is constrained exactly by the case $d=3$: we can write it in $1+2+1$ block form as

$$\left( \begin{array}{ccc}1 &v&s\\ 0&\gamma & v' \\ 0 & 0 & \omega \end{array} \right ) $$
with $\gamma \in \Gamma$.
We restrict to the subgroup $G_{-3}$ of finite index such that 
$\gamma = {\mathbf 1}_{\Gamma}$. 
As in Subsection \ref{ssdegnondeg}, writing that the form $Q_{\frac 14}$ is preserved gives (when $\gamma = {\mathbf 1}_{\Gamma}$)
$$ \omega = 1, \qquad v' = -^t \bar v, \qquad  \Re s = -\frac 12 ||v||^2 .$$

\begin{lemma}
Conversely, any matrix of the prescribed form satisfying these relations belongs
to $G_{-3}$.
\end{lemma}

\begin{proof}
Essentially the same as in Proposition \ref{propnotsplit} and Corollary \ref {Zar1}.
\end{proof}

The $5$-dimensional Lie algebra $\mathfrak g_{-3}$ of $G_{-3}$ is therefore the set of matrices of the form

$$ A=\left( \begin{array}{cccc}0&v_{-1} & v_1& i\,s \\ 0 &0&0&- \bar v_{-1} \\
                                            0 & 0 & 0 & -\bar v_1 \\ 0 & 0 & 0 & 0 \end{array} \right ) $$
with $v_{-1},\, v_1 \in \Cset$ and $s \in \Rset$.

Similarly, from the action of $L_{-3}, \, L_{-1}, \, L_1$, one obtains a Zariski closed group $G_3$ whose Lie algebra $\mathfrak g_3$ is the set  of  matrices of the form

$$B=  \left( \begin{array}{cccc}0& 0 & 0 & 0   \\ - \bar u_{-1}  &0&0& 0\\
                                            -\bar u_1 & 0 & 0 & 0  \\ i\, r & u_{-1} & u_1 & 0 \end{array} \right ) $$
with $u_{-1},\, u_1 \in \Cset$ and $r \in \Rset$.

\begin{lemma} \label{14lem3}
The smallest Lie algebra containing $\mathfrak g_{-3}$ and $\mathfrak g_3$ is the Lie algebra ${\mathfrak su}(Q_{\frak 14})$.
\end{lemma}
\begin{proof}
With $A$, $B$ as above, we have
$$ AB-BA = \left( \begin{array}{cccc}
-rs-v_{-1}\bar u_{-1} -v_1 \bar u_1& is u_{-1} & isu_1 & 0  \\ 
- ir \bar v_{-1} & v_{-1} \bar u_{-1} - u_{-1} \bar v_{-1}   &v_{1} \bar u_{-1} - u_{1} \bar v_{-1} & is \bar u_{-1} \\
-ir \bar v_1 & v_{-1} \bar u_{1} - u_{-1} \bar v_{1}  & v_{1} \bar u_{1} - u_{1} \bar v_{1}  & is \bar u_1 \\ 
0 & -ir v_{-1} & -ir v_1 & r\,s + u_{-1}\bar v_{-1} + u_1 \bar v_1 \end{array} \right ) $$

After adding appropriate elements $A'$, $B'$ of $\mathfrak g_{-3}$, $\mathfrak g_3$ respectively, the matrix $AB -BA + A'+B'$ is equal to

$$ C = \left( \begin{array}{cccc}
-rs-v_{-1}\bar u_{-1} -v_1 \bar u_1& 0 & 0 & 0  \\ 
0 & v_{-1} \bar u_{-1} - u_{-1} \bar v_{-1}   &v_{1} \bar u_{-1} - u_{1} \bar v_{-1} & 0 \\
0 & v_{-1} \bar u_{1} - u_{-1} \bar v_{1}  & v_{1} \bar u_{1} - u_{1} \bar v_{1}  & 0 \\ 
0 & 0 & 0 & r\,s + u_{-1}\bar v_{-1} + u_1 \bar v_1 \end{array} \right ) $$

For $u_{-1} = u_1 = v_{-1} = v_1 =0$, $r=s=1$, we get

$$ C = \left( \begin{array}{cccc}
-1 & 0 & 0 & 0  \\ 
0 & 0   &0 & 0 \\
0 & 0 & 0 & 0 \\ 
0 & 0 & 0 & 1 \end{array} \right ) $$

For $ r=s=v_{-1} = u_1 = 0$, $v_1 = u_{-1} =1$, we get

$$ C = \left( \begin{array}{cccc}
0 & 0 & 0 & 0  \\ 
0 & 0   &1 & 0 \\
0 & -1 & 0 & 0 \\ 
0 & 0 & 0 & 0 \end{array} \right ) $$

For $ r=s=v_{-1} = u_1 = 0$, $v_1 = i$, $u_{-1} =1$, we get

$$ C = \left( \begin{array}{cccc}
0 & 0 & 0 & 0  \\ 
0 & 0   &i & 0 \\
0 & i & 0 & 0 \\ 
0 & 0 & 0 & 0 \end{array} \right ) $$

For $ r=s=v_{-1} = u_{-1} = 0$, $v_1 = 1$, $u_{1} = i $, we get

$$ C = \left( \begin{array}{cccc}
i & 0 & 0 & 0  \\ 
0 & 0   &0 & 0 \\
0 & 0 & -2i & 0 \\ 
0 & 0 & 0 & i \end{array} \right ) $$

For $ r=s=v_{1} = u_{1} = 0$, $v_{-1} = 1$, $u_{-1} = i $, we get

$$ C = \left( \begin{array}{cccc}
i & 0 & 0 & 0  \\ 
0 & -2i   &0 & 0 \\
0 & 0 & 0 & 0 \\ 
0 & 0 & 0 & i \end{array} \right ) $$

These five matrices, together with $\mathfrak g_{-3}$ and $\mathfrak g_3$, span a $15$-dimensional vector space. As ${\mathfrak su}(Q_{\frac 14})$, being isomorphic to ${\mathfrak su} (3,1)$, has also dimension $15$, the lemma is proved.

\end{proof}

\begin{corollary} \label{corZar14}
The intersection of the subgroup generated by $L_{-3}$, $L_{-1}$, $L_1$ and $L_3$ with $SL (\Cset ^{\A_4})$ is Zariski dense in $SU(Q_{\frac 14})$.

\end{corollary}

This provides an appropriate starting point for the induction for $\alpha = 1/4$. The results in higher dimension follow.

%%%%%%%%%%%%%%%%%%%%%%%%%%%%%%%%%%%
\subsection{Exceptional case III: $\mathbf{\alpha = 1/6}$} \label{ss16}
%%%%%%%%%%%%%%%%%%%%%%%%%%%%%%%%

The proof is identical to the previous case $\alpha = 1/4$ due to the following fact, an improvement (in generality) on Lemma \ref{14lem1}

\begin{proposition}\label{propRirreduc}
For any $\alpha \in (0,1/2)$, any $d \geq 2$, such that $(d+1) \alpha$ is not an integer, there is no non trivial $\Rset$-subspace of $\CAD$ which is invariant under every $L_p$, $p \in \Ad$.
\end{proposition}

\begin{proof}
The kernel of $L_p - {\rm id}$ is the hyperplane 
$$ \mathcal H_p = \{ \zeta  \sum_{q\leq p} x_q + \sum_{q \geq p} x_q =0 \}.$$
The other eigenvalue of $L_p$ is equal to $-\zeta$ and is simple. The associated eigenspace is $\Cset e_p$.
We claim that the intersection $\bigcap_{p \in \Ad} \mathcal H_p$ is trivial (if $\zeta^{d+1} \ne 1$). Indeed, the intersection $\mathcal H_p \cap \mathcal H_{p+2}$ is contained in $\{ x_p = \zeta x_{p+2} \}$ for $p< d-1$, and the intersection 
$\mathcal H_{d-1} \cap \mathcal H_{1-d}$ is contained in $\{ \zeta x_{1-d} = \zeta^{-1} x_{d-1} \}$.
\par
Let $p \in \Ad$, and let $W$ be a $\Rset$-subspace of $\CAD$ which is invariant under $L_p$. If $W$ is not contained in $\mathcal H_p$, it contains $\Cset e_p$: indeed, if $w \in W$ has the form $h + te_p$ with $h \in \mathcal H_p$ and $t \ne 0$, then $w - L_p(w)  = (1+\zeta)t e_p $ belongs to $W$ and $\Cset e_p \subset W$ as $\zeta$ is not real.

\par
Assume now that $W$ is a $\Rset$-subspace of $\CAD$ which is invariant under 
all $L_p$, $p \in \Ad$. If $W$ is contained in $\mathcal H_p$ for every $p \in \Ad$, it is equal to $\{0 \}$. Otherwise, there exists $p \in \Ad$ such that $\Cset e_p \subset W$. 
For $q \ne p$, $e_p - L_q(e_p)$ is equal to $e_q$ or $\zeta e_q$, therefore we have also $\Cset e_q \subset W$. Then $W = \CAD$.
\end{proof}

%%%%%%%%%%%%%%%%%%%%%%%%%%%%%%%%%%%%%%%%%%%%%%%%%%%%%%%%%%%%%%%%%%%%%%%%%%%%%%%%

\end{document}